%% file: FMRS_arXivVersion26.tex
\documentclass[11pt,reqno]{amsart}

\usepackage{amsmath,mathtools}
\usepackage{amsthm}
\usepackage{amssymb}
\usepackage{graphicx}
\usepackage{verbatim}
\usepackage[dvipsnames]{xcolor}
\usepackage{tikz}
\usepackage{caption} 
\usepackage{fullpage}
\usepackage{hyperref}
\usepackage{soul}
\usepackage{mathrsfs}

\usepackage{todonotes}

\usepackage[percent]{overpic}
\usetikzlibrary{math} 

\usepackage{pgfplots} 

\theoremstyle{definition}

\theoremstyle{theorem}
\newtheorem{theorem}{Theorem}

\newtheorem*{conjA}{Conjecture}
\newtheorem{lemma}[theorem]{Lemma}
\newtheorem{prp}[theorem]{Proposition}
\numberwithin{theorem}{section}

\definecolor{sageblue}{rgb}{0,0,1}
\definecolor{sagered}{rgb}{1,0,0}
\definecolor{sagegreen}{rgb}{0.0, 0.5019607843137255, 0.0}
\definecolor{sageorange}{rgb}{1.0, 0.6470588235294118, 0.0}

\newcommand{\N}{\mathbb{N}}
\newcommand{\R}{\mathbb{R}}
\newcommand{\Q}{\mathbb{Q}}
\newcommand{\C}{\mathbb{C}}
\newcommand{\Hh}{\mathbb{H}}
\newcommand{\Z}{\mathbb{Z}}
\newcommand{\Li}{\operatorname{Li}}

\newcommand{\sign}{\operatorname{sign}}

\newcommand{\re}{\textnormal{Re}}
\newcommand{\im}{\textnormal{Im}}

\renewcommand{\=}{\;=\;}

\renewcommand{\Re}{\re}
\renewcommand{\Im}{\im}
\newcommand{\Mod}[1]{\ (\mathrm{mod}\ #1)}

\newcommand{\oor}{\overline{\overline r}}

\usepackage{color}

\newcommand{\bea}{\begin{equation*}\begin{aligned}}
\newcommand{\eea}{\end{aligned}\end{equation*}}

\newcommand{\beal}{\begin{equation}\begin{aligned}}
\newcommand{\eeal}{\end{aligned}\end{equation}}

\makeatletter
\@namedef{subjclassname@2020}{%
	\textup{2020} Mathematics Subject Classification}
\makeatother

\theoremstyle{remark}
\newtheorem*{remark}{Remark}
\newtheorem*{remarks}{Remarks}

\title{Oscillating asymptotics and conjectures of Andrews}
\begin{document}
	
	\author{Amanda Folsom}
	\address{Department of Mathematics, Seeley Mudd Building, Amherst College, Amherst, MA, USA, 01002}
	\email{afolsom@amherst.edu}
	
\author{Joshua Males}
\address{School of Mathematics, University of Bristol, Bristol, BS8 1TW, UK, and the Heilbronn Institute for Mathematical Research, Bristol, UK.}
\email{joshua.males@bristol.ac.uk}
	
	\author{Larry Rolen}
	\address{Department of Mathematics, 
1420 Stevenson Center, 
Vanderbilt University, 
Nashville, TN, USA, 37240}
	\email{larry.rolen@vanderbilt.edu}

	\author{Matthias Storzer}
\address{School of Mathematical Sciences,
	University College Cork, Cork, Ireland}
\email{mstorzer@ucc.ie}
	
	
	
	\thanks{The first author is   grateful for support from NSF Grant DMS-2200728, and also thanks the Max  Planck Institute for Mathematics, Bonn, Germany,  for its hospitality and support during portions of the writing of this paper. The research of the second author conducted for this paper was partially supported by the Pacific Institute for the Mathematical Sciences (PIMS). The research and findings may not reflect those of the Institute.  
	This work was supported by a grant from the Simons Foundation (853830, LR). The third author is also grateful for support from a 2021-2023 Dean's Faculty Fellowship from Vanderbilt University. 
	The fourth author has been supported by the Max-Planck-Gesellschaft.
	The authors also acknowledge the Vanderbilt University ``100 Years of Mock Theta Functions:
New Directions in Partitions, Modular Forms, and Mock Modular Forms'' Conference in May 2022, at which they discussed the topic of this paper and related ideas. This conference was supported by The Shanks Endowment,  Vanderbilt University, NSF Grant: “100 Years of Mock Theta Functions; New Directions in Partitions, Modular Forms, and Mock Modular Forms” award number: DMS-1951393, and NSA grant “100 Years of Mock Theta Functions,” award number: H98230-20-1-0022.}
	 
\maketitle

\begin{abstract}
In 1986, Andrews \cite{Andrews86, AndrewsAdv} studied the function $\sigma(q)$ from Ramanujan's ``Lost" Notebook, and made several conjectures on its Fourier coefficients $S(n)$, which count certain partition ranks.  In 1988, Andrews-Dyson-Hickerson \cite{ADH} famously resolved these conjectures, relating the coefficients $S(n)$ to the arithmetic of $\mathbb Q(\sqrt{6})$;  this relationship was further expounded upon by Cohen \cite{Cohen} in his work on Maass waveforms, and was more recently extended by Zwegers \cite{ZwegersWave} and by Li and Roehrig \cite{LR}. A closer inspection of Andrews' original work on $\sigma(q)$ reveals additional related functions and conjectures, which we study in this paper.  In particular, we study the function $v_1(q)$, also from Ramanujan's ``Lost" Notebook,  a $q$-hypergeometric series with partition-theoretic Fourier coefficients $V_1(n)$, and prove two of Andrews' conjectures on $V_1(n)$ which are parallel to his original conjectures on $S(n)$.   Our methods differ from those used in \cite{ADH}, and  require a blend of novel techniques inspired by Garoufalidis' and Zagier's recent work on asymptotics of Nahm sums \cite{GZ, GZknots}, with classical techniques including the Circle Method in Analytic Number Theory; our methods may also be applied to determine the asymptotic behavior of similar $q$-hypergeometric series of interest which are not amenable to classical techniques. We also offer explanations of additional related conjectures of Andrews, ultimately connecting the asymptotics of $V_1(n)$ to the arithmetic of $\mathbb Q(\sqrt{-3})$.
\end{abstract}

\section{Introduction and statement of results} \label{sec_intro}
In \cite{ADH}, Andrews, Dyson, and Hickerson famously studied the $q$-hypergeometric series
\begin{align*}
\sigma(q) &:= \sum_{n=0}^\infty \frac{q^{\frac{n(n+1)}{2}}}{(-q;q)_n} 
=: \sum_{n=0}^\infty S(n) q^n,
\end{align*}
found in Ramanujan's ``Lost" Notebook, along with its companion 
\begin{align*}
\sigma^*(q) &:= 2\sum_{n=0}^\infty \frac{(-1)^n q^{n^2}}{(q;q^2)_n}, 
\end{align*}
which was discovered later. Here and throughout, we let $q \coloneqq e^{2\pi i \tau}$ with $\tau \in \Hh$, the upper half-plane. Moreover, we define the $q$-Pochhammer symbol by
$(a;q)_n = \prod_{j=0}^{n-1}(1-aq^j)$ for $n \in \mathbb N_0\cup \{\infty\}.$  On one hand, these functions can be interpreted combinatorially, e.g., when expanded as a $q$-series, the coefficients $S(n)$ of $\sigma(q)$ count the difference between the number of partitions into distinct parts with even and odd rank.   The authors of \cite{ADH} were in part motivated to study these functions following Andrews' earlier conjectures \cite{Andrews86}:
\begin{conjA}[Conjecture 1 \cite{Andrews86}]
$\limsup |S(n)| = +\infty.$
\end{conjA}
\begin{conjA}[Conjecture 2 \cite{Andrews86}] 
$S(n)=0$ for infinitely many $n$.
\end{conjA}\noindent By showing   a deep connection between these hypergeometric series  and the arithmetic of $\mathbb Q(\sqrt{6})$, extending beyond their combinatorial interpretations, Andrews-Dyson-Hickerson \cite{ADH} succeeded in proving Andrews' two conjectures on $\sigma(q)$ above.  For example, we now know that 
 the 
coefficients of $\sigma(q)$ may also be defined   by a Hecke $L$-function which is a certain sum over ideals in $\mathbb Z[\sqrt{6}]$ \cite{ADH, Cohen}.    In an extension of the work in \cite{ADH}, Cohen \cite{Cohen} further constructed a Maass waveform from (the coefficients of) $\sigma$ and $\sigma^*$, which 
Zwegers generalized in \cite{ZwegersWave}, by constructing a family of what he termed mock Maass theta functions  associated to indefinite binary quadratic forms.  This was used to provide further examples of similar $q$-series related to the Maass waveforms in \cite{BringmannLovejoyRolen,KrauelRolenWoodbury,LiNgoRhoades}. It has very recently been reconstructed using theta integrals by Li and Roehrig \cite{LR}, who used this as a motivational example to discover new real-analytic modular forms whose Fourier coefficients are given by logarithms of real quadratic numbers.

In the same paper \cite{Andrews86}, Andrews made further conjectures 
like those above for $\sigma(q)$  for another function $v_1(q)$ also from Ramanujan's ``Lost" Notebook, defined by
\begin{equation}
	v_1(q) := \sum_{n\geq0} \frac{q^{n(n+1)/2}}{(-q^2;q^2)_n}
	\;=:\; \sum_{n\geq0} V_1(n)q^n.
\end{equation}  The function $v_1(q)$ admits a similar combinatorial interpretation to $\sigma(q)$.
 A partition is called odd-even if the parity of the parts alternates, starting with the smallest part odd.
The rank, defined as the difference between the largest part and the number of parts, of an odd-even partition is even.
The 
 coefficients $V_1(n)$ count the difference between the number of odd-even partitions of $n$ with rank congruent to $0\pmod{4}$ and congruent to $2\pmod{4}$ (see also \cite{AndrewsLostIV}).

While he noted that the growth of $|V_1(n)|$ ``is not very smooth," Andrews conjectured that there ``appear[s] to be great sign regularity". More precisely, he states the following conjectures.
\begin{conjA}[Conjecture 3 \cite{Andrews86}]
We have that $|V_1(n)|\to \infty$ as $n\to \infty$.
\end{conjA}
\begin{conjA}[Conjecture 4 \cite{Andrews86}]
For almost all $n$, $V_1(n), V_1(n+1), V_1(n+2)$ and $V_1(n+3)$ are two positive and two negative numbers.
\end{conjA}
\begin{conjA}[Conjecture 5 \cite{Andrews86}]
For $n \geq 5$  there is an infinite sequence $N_5 = 293, N_6 = 410, N_7 = 545, N_8 = 702,...$, $N_n > 10n^2,...$ such that $V_1(N_n), V_1 (N_n + 1), V_1 (N_n + 2)$ all have the same sign.
\end{conjA}
\begin{conjA}[Conjecture 6 \cite{Andrews86}]
With reference to Conjecture 3, the numbers $|V_1(N_n)|,
|V_1(N_n + 1)|, |V_1(N_n + 2)|$ contain a local minimum of the sequence $|V_1(j)|$. 
\end{conjA}

While of interest in their own right, we now know thanks in particular to work in \cite{ADH, Cohen, ZwegersWave}   that $\sigma(q)$ and related functions are also of interest due to connections to real quadratic fields and Maass waveforms;  Andrews himself also put  $\sigma(q), v_1(q)$ and related functions into a broader context by asking more general questions about $q$-series with bounded and unbounded coefficients in \cite{Andrews86}.    Our main results in this paper are as follows.

\begin{theorem}\label{thm_main}
Andrews' Conjecture~4 is true.
Moreover, Andrews' Conjecture~3 is true in the following sense: $|V_1(n)|\to\infty$ as $n\to\infty$ away from a set of density $0$.
\end{theorem} 

\begin{remark}   
After computational and theoretical investigations, we have modified 
Andrews' Conjecture~3
 to say ``$|V_1(n)|\rightarrow\infty$ as $n\rightarrow \infty$ away from a set of density $0$" as above.    This refined conjecture will follow from the arguments needed to prove Andrews' Conjecture 4  (see  Section \ref{Sec: proof of Andrew's conjectures}).
\end{remark}
\begin{remark} 
  See Section \ref{Sec: explanations} for our explanations of Andrews' Conjecture~5 and Conjecture~6, ultimately relating the asymptotics of $V_1(n)$ to the arithmetic of $\mathbb Q(\sqrt{-3})$.
\end{remark}

\begin{figure}[ht] 
\centering
\begin{overpic}[scale=0.6]
{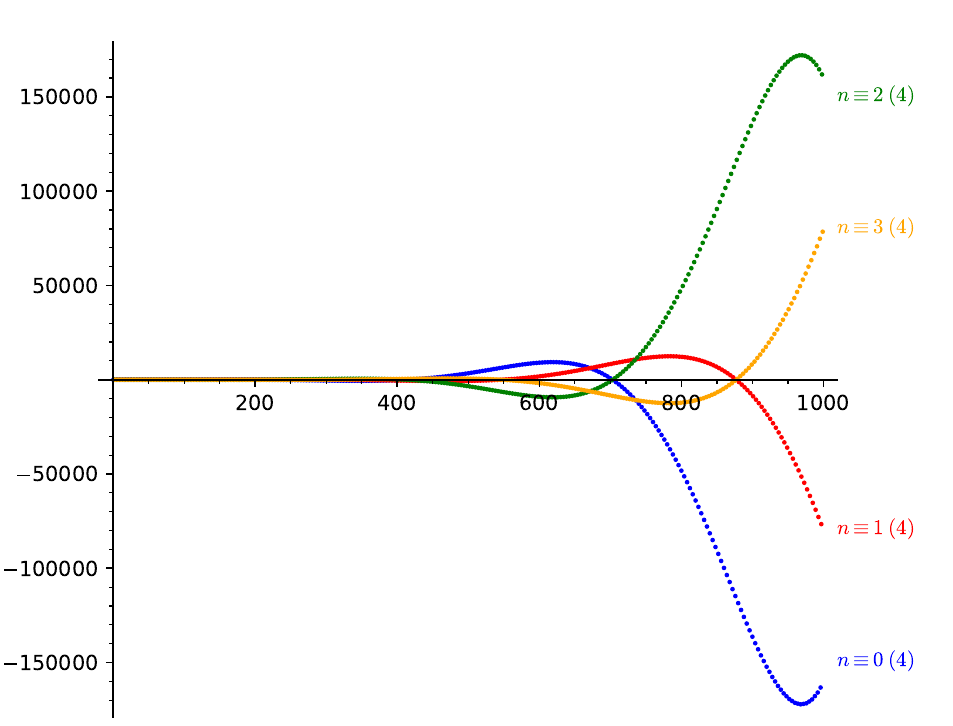}
\put(48,40){\scalebox{.5}{
	$\textcolor{sageblue}{+}
	\textcolor{sagered}{-}
	\textcolor{sagegreen}{-}
	\textcolor{sageorange}{+}$}}
\put(58,40){\scalebox{.5}{
	$\textcolor{sageblue}{+}
	\textcolor{sagered}{+}
	\textcolor{sagegreen}{-}
	\textcolor{sageorange}{-}$}}
\put(69,40){\scalebox{.5}{
	$\textcolor{sageblue}{-}
	\textcolor{sagered}{+}
	\textcolor{sagegreen}{+}
	\textcolor{sageorange}{-}$}}
\put(82,40){\scalebox{.5}{$
	\textcolor{sageblue}{-}
	\textcolor{sagered}{-}
	\textcolor{sagegreen}{+}
	\textcolor{sageorange}{+}$}}
\end{overpic}
\caption{$V_1(n)$ for $n=1,\ldots,1000$}
\label{coeffs_V_1}
\end{figure}
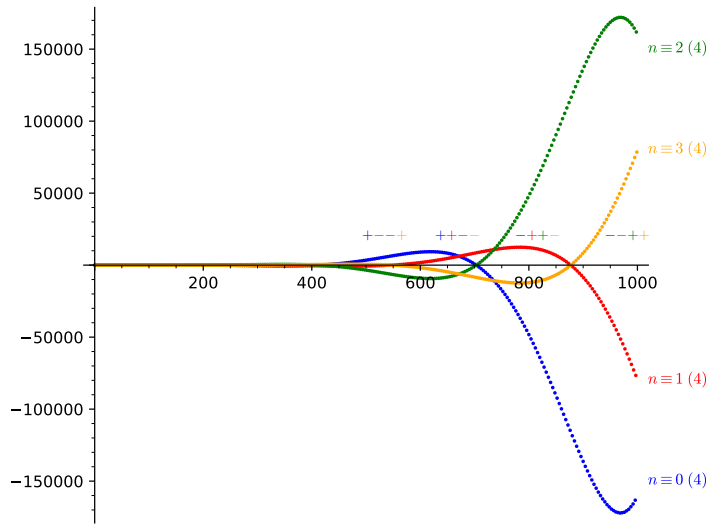

Figure \ref{coeffs_V_1} shows the first $1000$ values for $V_1(n)$ together with their sign patterns; one can see that the asymptotics of $V_1(n)$ appears to depend on $n\pmod 4$. Moreover, the sequence can be divided into $4$ reoccurring sections with the following patterns for $\sign(V_1(n)),\ n\equiv n_0\pmod 4$.
\begin{center}
\begin{tabular}{ c | c c c c}
$n_0$ & $0$ & $1$ & $2$ & $3$\\ \hline
Section $1$ & $+$ & $-$ & $-$ & $+$\\
Section $2$ & $+$ & $+$ & $-$ & $-$\\
Section $3$ & $-$ & $+$ & $+$ & $-$\\
Section $4$ & $-$ & $-$ & $+$ & $+$
\end{tabular}
\end{center}
For example, in Figure \ref{coeffs_V_1}, we see the sign pattern $+--+$ between $n=546$ and $n=702$ (Section $1$), the sign pattern $++--$ between $n=703$ and $n=877$ (Section 2), etc. We ultimately establish the sign regularity of $V_1(n)$ in Theorem \ref{thm_main} after establishing the asymptotic behavior of  $v_1(q)$, our second main result.

\subsection{The asymptotics of $v_1$} Here and throughout, we let $e(u): = e^{2\pi i u}$. We make extensive use of the dilogarithm, which is a natural extension of the usual logarithm function, and can be defined by the power series
\begin{align*}
	\Li_2(z) \coloneqq \sum_{n \geq 1} \frac{z^n}{n^2}
\end{align*}
for $|z|<1$, and naturally extended to the cut plane $\C \setminus [1,\infty)$ by analytic continuation, yielding
\begin{align*}
	\Li_2(z) = - \int_{0}^{z} \log(1-u) \frac{du}{u}
\end{align*}
for $z \in \C \setminus [1,\infty)$. We note that we also use a slightly different branching of $\Li_2$ for the proof of Theorem \ref{Thm: radial asymp}, see Section \ref{Sec: proof of conj 1}.

 To state our results, we require the Bloch-Wigner dilogarithm (see e.g.\@ \cite{Zdilog})
\begin{align*}
	\operatorname D(z) \coloneqq \im(\Li_2(z)) +\arg(1-z)\log|z|.
\end{align*}

\begin{theorem}\label{Thm: radial asymp} Let $\zeta = e(\alpha)\in\C$ be a root of unity of order ${ m}\in\N_{0}$. 
	\begin{enumerate}
		\item If $4\nmid { m}$, then $v_1(\zeta e^{-z})=O(1)$ as $z\to 0$ along any ray in the right half-plane.
		\item If $4 | { m}$, then as $z\to 0$, on a ray in the right half-plane with $0\neq |\arg{z}|< \frac \pi 2$, we have 
\beal\label{eq:asymp_v1}
v_1(\zeta e^{-z})
\=&
\exp\left({\frac { 16V}{zm^2}}\right)\;
 \sqrt{\frac {2\pi i}{z}}
\gamma_{(\alpha)}^+(z){(1+O(|z|^L))}\\
&\;+\;
\exp\left({\frac {- 16V}{zm^2}}\right)\;
 \sqrt{\frac {2\pi i}{-z}}
\gamma_{(\alpha)}^-(z)
{(1+O(|z|^L))}
\\
\eeal
for all $L>0$, 
where $V$ is given in terms of the Bloch-Wigner dilogarithm $\operatorname D$ by
\bea
V &\= \operatorname D (e(1/6)) \frac i 8 \= 0.1268677\ldots i,
\eea
and the power series $\gamma^{\pm}_{(\alpha)}(z)\in\C[[z]]$ are defined in Section \ref{sec:pfasympv12}.
\item In particular, for $\zeta = \pm i$, the power series are defined in \eqref{eq:defgammam}, resp. \eqref{eq:defgammam2} and the first terms are given by  
\begin{align*}
&\gamma^{+}_{(1/4)}(z) \= 
\overline{\gamma^{-}_{(3/4)}(z)} \= 
\gamma^+
\left(
1+
\left(\frac 1 3-\frac{77}{216}\sqrt3 \right)iz
+\left(-\frac{89449}{31104}+\frac{647}{648}\sqrt 3\right)z^2+O(|z|^3)
\right),\\
&\gamma^{-}_{(1/4)}(z) \= 
\overline{\gamma^{+}_{(3/4)}(z)} \= 
\gamma^-
\left(
1+
\left(\frac 1 3+\frac{77}{216}\sqrt3 \right)iz
+\left(-\frac{89449}{31104}-\frac{647}{648}\sqrt 3\right)z^2+O(|z|^3)
\right),
\end{align*}
where
\begin{align*}
\gamma^+ \;\coloneqq\;
	\gamma^{+}_{(1/4)}(0)
	\=\gamma^{-}_{(3/4)}(0)
	&\= \frac {1}{2\sqrt[4]{3(2-\sqrt{3})}}
	\= 0.5280518 \ldots,\\
\gamma^- \;\coloneqq\;
	\gamma^{-}_{(1/4)}(0)
	\= 	\gamma^{+}_{(3/4)}(0)
	&\= \frac 1 {2\sqrt [4] {3(2+\sqrt 3)}}
	\=0.2733397 \ldots.
\end{align*}
	\end{enumerate}
\end{theorem}

	\begin{remarks} \,
		\begin{enumerate}
	\item We have that $V = \frac{\mathcal{G}i}{8},$ where $\mathcal{G} = 1.0149...$ is the maximum of the Bloch-Wigner dilogarithm function $\operatorname D(z)$, attained at $e(1/6)$ \cite{Zdilog}.    We note that $\mathcal G$ is also the minimal hyperbolic volume of non-compact hyperbolic 3-manifolds. Since this volume is attained by the Gieseking manifold \cite{colin}, $\mathcal G$ is also known as Gieseking's constant.  
\item Note that when $z$ approaches zero on a given ray in the right half-plane with $\arg(z)>0$, the first term in \eqref{eq:asymp_v1} will be exponentially large while the other one will be exponentially small and vice versa if $\arg(z)<0$.
		\end{enumerate}
	
	\end{remarks}

\subsection{The asymptotics of $V_1$}

We use Theorem \ref{Thm: radial asymp} to ultimately establish the following theorem giving the first term asymptotic approximation for $V_1(n)$.

\begin{theorem}\label{Conj3} 
	As $n \to \infty$ we have
		\begin{align}\label{v1nasy-new1}
		V_1(n)
		\=&(-1)^{\lfloor \frac n 2 \rfloor}\ 
		\frac {e^{\sqrt{2|V|n}}} {\sqrt{n}}
		{(\gamma^+ + (-1)^n\gamma^-)}
		\left(\cos(\sqrt{2|V| n}) -{{(-1)^{n}}} \sin(\sqrt{2|V| n})\right) \left(1+O\left(n^{-\frac{1}{2}} \right)\right) \notag\\
		& + O\left( n^{-\frac{1}{2}} e^{\sqrt{\frac{|V|n}{2}}} \right).
	\end{align} 
\end{theorem}

The sequence $V_1(n) {e^{-\sqrt{2|V|n}}} {\sqrt{n}}$ is plotted in Figure \ref{figure_approx}.
\begin{figure}[ht]
	\centering
	\includegraphics[scale=0.6]{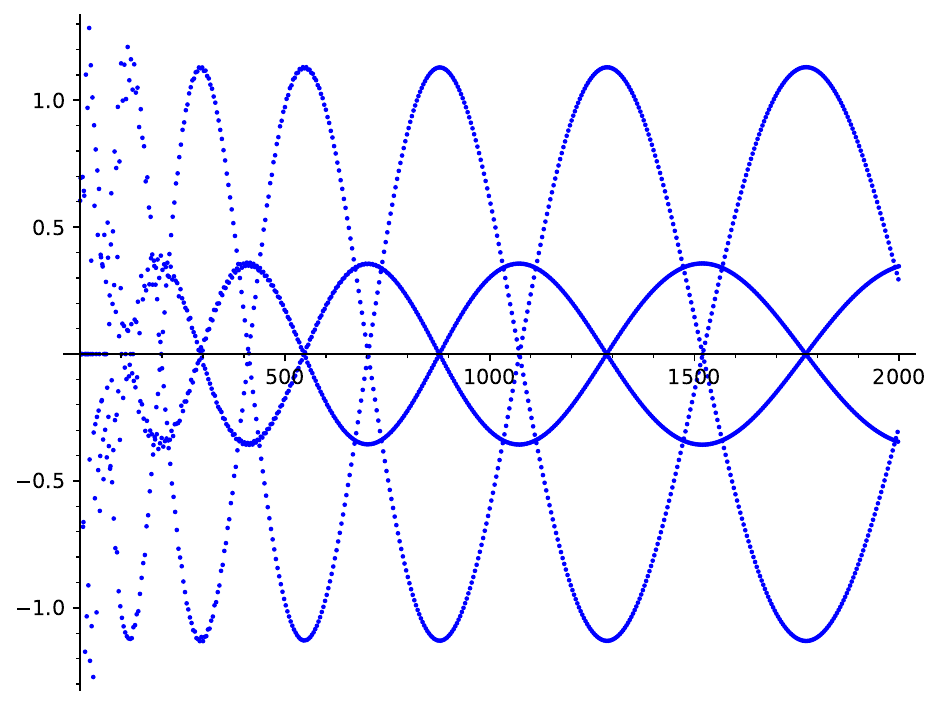}
	\caption{$V_1(n) {e^{-\sqrt{2|V|n}}} {\sqrt{n}}$ for $0\leq n\leq2000$}
	\label{figure_approx}
\end{figure}

From Figures \ref{coeffs_V_1} and \ref{figure_approx} it appears that $V_1(n)$ is arbitrarily small infinitely often. We discuss this observation alongside  Andrews' Conjecture~5 and Conjecture~6 in Section \ref{Sec: explanations}, giving candidates for these points and providing strong numerical and heuristic evidence.

In fact, our methods give a stronger asymptotic for $V_1(n)$, namely, the (non-convergent) asymptotic formula 
\begin{multline}\label{eqn:V1n-asymp1}
	V_1(n) = \sum_{m \geq 1}   \sum_{\substack{k = 1 \\ \gcd(k,4m)=1}}^{4m} \frac{\zeta_{4m}^{-kn}}{\sqrt{2i} \sqrt{n}} \exp\left( \frac{2 \sqrt{Vn}}{m} \right) \gamma^+_{\left(\frac{k}{4m}\right)} \left( \frac{1}{m} \sqrt{\frac{V}{n}} \right) \left(1+O\left(n^{-\frac{1}{2}}\right)\right) \\
	+ \frac{\zeta_{4m}^{-kn}}{\sqrt{-2i} \sqrt{n}} \exp\left( \frac{2 \sqrt{-Vn}}{m} \right) \gamma^-_{\left(\frac{k}{4m}\right)} \left( \frac{1}{m} \sqrt{\frac{-V}{n}} \right) \left(1+O\left(n^{-\frac{1}{2}}\right)\right).
\end{multline}
Our (non-convergent) asymptotic in \eqref{eqn:V1n-asymp1} is reminiscent of the original formula of Hardy and Ramanujan for $p(n)$ before it was extended to Rademacher's celebrated exact formula for $p(n)$, as well as the famous Andrews-Dragonette formula for the coefficients of the mock theta function $f(q)$ \cite{AndrewsDragonette}, and their conjectured exact formula which was proved by Bringmann and Ono in their breakthrough paper \cite{BO}. Rademacher's beautiful exact formula for $p(n)$, as it turns out, is intimately connected to the theory of Maass-Poincar\'e series.  It would be of interest to explore similar formulae and connections   for $V_1(n)$ related to \eqref{eqn:V1n-asymp1}. 

To establish \eqref{eqn:V1n-asymp1}, our calculations are as in Section \ref{Sec: major arc} using the Circle Method and collecting contributions from all roots of unity as we now explain. Let $\zeta_{4m}^{k} \coloneqq e^{\frac{2\pi i k}{4m}}$. First note that as $m \to \infty$ the $4m$-th roots of unity become dense on the unit circle (and the other roots of unity do not give any growing asymptotic contribution by Theorem \ref{Thm: radial asymp}), and that the asymptotic formulae given in Theorem \ref{Thm: radial asymp} hold in cones across each root of unity. In turn, this means that we simply need to sum the contributions from each $4m$-th root of unity to obtain the claimed asymptotic behavior, and there is no error term being bound.
	
	At each root of unity $\zeta_{4m}^{k}$ with $\gcd(k,4m) =1$ (so that it is primitive), we simply repeat the arguments given in Section \ref{Sec: major arc}, but instead considering an arc around $\zeta_{4m}^k$ parametrized by $q = \zeta_{4m}^k e^{-\lambda + i \theta}$ with $\theta \in (-\delta, \delta)$. After plugging in the asymptotic for $v_1(\zeta_{4m}^k e^{-z})$ given in Theorem \ref{Thm: radial asymp}, the same calculations using the saddle-point method then give \eqref{eqn:V1n-asymp1}.
	
Extracting the $m=1$ term and combining the coefficients yields the first-term approximation given in Theorem \ref{Conj3}.
Of course, since $\gamma_\cdot^\pm$ are power series, the error terms make this asymptotic lose some meaning. This could be remedied by developing techniques to control the radial asymptotic to arbitrary precision. In particular one would need much stronger methods than the saddle point method provides us.

 Returning to our main results, we remark that our  methods used to prove Theorem \ref{thm_main} differ from the methods used in \cite{ADH} to prove Andrews' related conjectures on $\sigma(q)$, which are not obviously applicable; however, it would be of interest to find a $q$-series identity for $v_1(q)$ using Bailey pairs or other methods  which lends itself to revealing more information about the behavior of $V_1(n)$ in an  analogous way.  Instead, our methods are inspired by both newer methods of Garoufalidis and Zagier on asymptotics of Nahm sums \cite{GZ, GZknots}, and older methods in Analytic Number Theory including Wright's Circle Method and the saddle-point method. In particular, standard techniques in the literature are not well-suited to studying the radial asymptotic behavior of $v_1$. Instead, in order to prove Theorem \ref{Thm: radial asymp} we determine an integral representation of $v_1$ which is similar to Watson's contour integral \cite{GR, Watson1910}. Our novel approach has applications beyond simply the study of the function $v_1$; in particular it can be used to compute the asymptotic behavior of  similar $q$-hypergeometric series  which are not amenable to classical techniques.

The aforementioned sums named after Nahm were introduced in \cite{Nahm} in relation to characters of rational conformal field theories.

They are a special class of $q$-hypergeometric series \cite{GR} given in the one-dimensional case by
$\sum_{n \geq 0}   {q^{An^2/2+bn+c}}/{(q;q)_n},$ 
for $A\in\Q_{>0}, b,c\in\Q$.
 In part motivated by the famous Rogers-Ramanujan identities \cite{AndrewsThy}, Zagier classified all one-dimensional modular Nahm sums in \cite{Zdilog}, thereby proving a conjecture of Nahm in the one-dimensional case. 
 
 These sums, in addition to having natural applications elsewhere, have previously-studied asymptotic properties which inform our study of Andrews' series $v_1(q)$ here. For example,  Nahm sums come equipped with a Nahm equation; in the one-dimensional case the equation $1-X=X^A$, whose solutions are conjecturally  connected to the modularity of the Nahm sum at hand \cite{CGZ,VZ,Zdilog}. In particular, the dilogarithm of the unique solution in $(0,1)$ determines the asymptotic behavior of the Nahm sum as $q$ approaches a root of unity.
 
 The $q$-hypergeometric series $v_1(q)$ studied here is nearly a Nahm sum, up to the change in sign appearing in the $q$-Pochhammer symbols $(-q^2;q^2)_n$ in the denominators of its summands.   For $v_1(q)$, the  analogue of the  Nahm-equation becomes $(1-Q)^2=-Q$, see \eqref{eq:nahmequation} with $Q=e^{-4iv_0}$, and the dilogarithms of its solutions $e(\pm 1/6)$ appear in the asymptotics of $v_1(q)$ in Theorem~\ref{Thm: radial asymp}.
In particular, Theorem~\ref{Thm: radial asymp} implies that $v_1(q)$ is not a modular form. For more on these $q$-series and related recent work, see the previously mentioned \cite{GZ, GZknots,Sthesis, Zdilog} and \cite{Nahm}, as well as work of Calegari-Garoufalidis-Zagier \cite{CGZ} and Vlasenko-Zwegers \cite{VZ}, and their references.  

We further note that Andrews also conjectures in \cite{Andrews86} that the coefficients of three additional functions have behavior similar to that of the sequence $V_1(n)$. It is natural to expect that the analysis and methods of proof given here for $v_1(q)$ and $V_1(n)$ could also be applied to study the three additional functions in \cite{Andrews86}, as well as other similar functions (e.g., see \cite{21LovejoyCologne}). Also note that $v_1$ is closely related to the function $\mathscr{O}$ studied by Jang in \cite{MJ}.

The remainder of the paper is structured as follows. In Section \ref{sec_prelim}, we  provide some background on Wright's Circle Method and the saddle-point method used later in the paper, and also 
establish some preliminary asymptotic results in the spirit of recent work of Garoufalidis and Zagier.  In Section~\ref{Sec: proof of conj 1} we prove Theorem \ref{Thm: radial asymp} on the asymptotics of $v_1(q)$.  In Section \ref{sec_3} we prove Theorem \ref{Conj3} on the asymptotics of $V_1(n)$.  In Section \ref{Sec: proof of Andrew's conjectures}, we prove Theorem \ref{thm_main} (Andrews' Conjectures 3 and 4). Finally, in Section \ref{Sec: explanations} we offer explanations for Andrews' Conjectures~5 and~6 regarding the coefficients $V_1(n)$, ultimately relating them to the arithmetic of $\mathbb Q(\sqrt{-3})$.

\section*{Acknowledgments}
The authors would like to thank Kathrin Bringmann, Jeremy Lovejoy, Andreas Mono, Campbell Wheeler and Don Zagier for enlightening conversations and helpful comments on the paper.
 Moreover, the authors would like to thank the referee for their careful reading and helpful suggestions. 
\section{Preliminaries}\label{sec_prelim}

\subsection{The saddle-point method}\label{sec_sp} To asymptotically estimate certain integrals that appear in our work, we make use of the saddle-point method. We use the ideas presented by Olver \cite{Olver}, O'Sullivan \cite{OSullivan}, as well as the notes \cite{Kap}. For convenience, we recall its essence here. Consider an integral of the shape
 \begin{align}\label{eqn: saddle integral}
 	\int_\Gamma f(z) e^{A g(z)} dz
 \end{align}
where $f,g$ are complex analytic functions and $\Gamma$ is a contour in the complex plane. We wish to approximate the integral as $A \to \infty$.

Since $f,g$ are analytic, we are able to continuously deform the contour $\Gamma$ without changing the value of the integral (if one instead deforms the contour over poles of the integrand, one simply needs to take care to include residues). The points where the real part of $g(z)$ is maximized and the imaginary part of $g(z)$ is constant are called {\it saddle-points}, and are zeros of $g'(z)$.

By shifting the path $\Gamma$ to a path running through the saddle-point and making appropriate shifts of the integration variable to center on the zero of $g(z)$, the integral \eqref{eqn: saddle integral} may be rewritten in terms of Gaussian-like integrals. These integrals may then be approximated by well-known means for large values of $A$.

\subsection{Wright's Circle Method}\label{Sec: WCM prelim}
Wright \cite{Wright1,Wright2} developed a modified version of Hardy and Ramanujan's Circle Method. Wright's work provides a very general approach to obtaining the asymptotic behavior of Fourier coefficients of generating functions whose radial asymptotic behavior towards roots of unity is known. We recall the essence of Wright's Circle Method here for the convenience of the reader. 

Consider a generating function $f(q) \coloneqq \sum_{n \geq 0} A(n) q^n$ with radius of convergence equal to one. The central idea is to use Cauchy's integral theorem to recover the coefficients $A(n)$ as
\begin{align*}
	A(n) = \frac{1}{2\pi i} \int_C \frac{f(q)}{q^{n+1}} dq,
\end{align*}
where $C$ is a circular contour of radius less than one, transversed precisely once in the anticlockwise direction. Since we are able to choose the radius of $C$, we pick a radius that tends to $1$ as $n \to \infty$ (meaning that $C$ tends to the unit circle).  Now assume that in arbitrarily wide cones inside the unit disk toward roots of unity the generating function $f(q)$ has known asymptotic behavior.
We then choose to place so-called major arcs around the roots of unity where $f(q)$ has the largest growth, and so-called minor arcs everywhere else. Using varying asymptotic methods, one is able to precisely determine the contribution of the major arcs to the asymptotic growth of the coefficients $A(n)$, which we call the main term and denote by $M(n)$. In the present paper, we use the saddle-point method. On the other hand, contributions from the minor arc are bounded more crudely in an error term, denoted $E(n)$. Overall this yields an asymptotic of the form $A(n) = M(n) + E(n)$.

Although Wright's Circle Method gives weaker bounds than the original version of Hardy and Ramanujan (and loses the possibility of exact formulae), it is much more flexible for dealing with non-modular generating functions, and has seen extensive use in the literature in recent years.

\subsection{Asymptotics}
Using the usual Bernoulli numbers $B_0=1$, $B_1=-\frac 1 2$, $B_2=\frac 1 6,\ldots$, the Bernoulli polynomials are defined for $n\in \mathbb N_0$ by
\bea
	B_n(X) := \sum_{k=0}^n \binom n k B_{k} X^{n-k}.
\eea

	 For $\varphi\in\C, |\varphi|=1, |\arg \varphi|<\frac \pi 2$ let $\widetilde \log$ be the logarithm such that $\Li_1^\varphi (e^{-iv}) = -\widetilde\log(1-e^{-iv})$ has branch cuts whenever $\re((v+2\pi n)/\varphi) = 0$ for some $n\in\Z$ and $\re(v)>0$. In other words, the principal branch cuts are rotated by $\varphi$, cf. Figure \ref{fig:branchingLi2tilde}. 
	For $v\in\C$ not on a branch cut we define the dilogarithm
	\bea
		\Li_2^\varphi(e^{-iv}) \= 
		\int_{0}^\infty
		\Li_1^\varphi\left(e^{-\varphi x-iv}\right)
		dx\\
	\eea
	where we avoid the branch cuts of $\Li_1^\varphi\left(1-e^{-\varphi x-iv}\right)$ (cf. Figure \ref{fig:branchingLi2tilde}).
	Then $\Li_2^{\varphi}(e^{-iv})$ has the same branch cuts as $\Li_1^{\varphi}(e^{-iv})$ and jumps by $2\pi v$ when $v$ crosses a branch cut (cf. \cite{Zdilog}).
	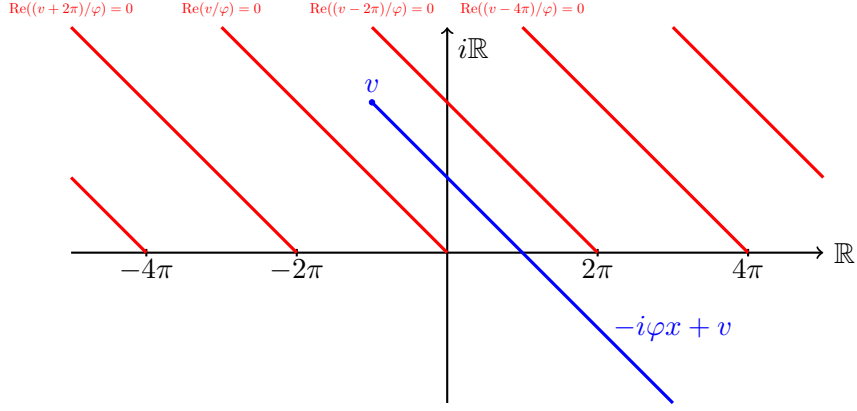
\begin{figure}
	\centering
	\begin{tikzpicture}
		\draw[black, thick,->] (-5,0) -- (5,0)	node[anchor=west]{$\R$};
		\draw[black, thick] (2,-.05) --  (2,.05) node[anchor=north]{$2\pi$};
		\draw[black, thick] (4,-.05) --  (4,.05) node[anchor=north]{$4\pi$};
		\draw[black, thick] (-2,-.05) --  (-2,.05) node[anchor=north]{$-2\pi$};
		\draw[black, thick] (-4,-.05) --  (-4,.05) node[anchor=north]{$-4\pi$};
		\draw[black, thick,->] (0,-2) -- (0,3)	node[anchor=north west]{$i\R$};
		\draw[red, very thick] (0,0) -- (-3,3);
		\node[red, anchor=south] at (-3,3) {\scalebox{.5}{$\re(v/\varphi) = 0$}};
		\draw[red, very thick] (2,0) -- (-1,3);
		\node[red, anchor=south] at (-1,3) {\scalebox{.5}{$\re((v-2\pi)/\varphi) = 0$}};
		\draw[red, very thick] (4,0) -- (1,3);
		\node[red, anchor=south] at (1,3) {\scalebox{.5}{$\re((v-4\pi)/\varphi) = 0$}};
		\draw[red, very thick] (5,1) -- (3,3);
		\draw[red, very thick] (-2,0) -- (-5,3);
		\node[red, anchor=south] at (-5,3) {\scalebox{.5}{$\re((v+2\pi)/\varphi) = 0$}};
		\draw[red, very thick] (-4,0) -- (-5,1);
		\draw[blue, very thick] (3,-2) -- (-1,2);
		\filldraw[blue] (-1,2) circle (1pt) node[blue, anchor=south] at (-1,2) {$v$};
		\node[blue] at (3,-1) {$-i\varphi x +v$};
		node[anchor=west]{Intersection point};
	\end{tikzpicture}
	\caption{The branch cuts of $\Li_s^\varphi(e^{-iv})$, $s=1,2$ and $-i\varphi x +v$ for $x\geq 0$}
	\label{fig:branchingLi2tilde}
	\end{figure}

Moreover, for $m\in\Z_{\leq{0}}$ the polylogarithm $\Li_m$ is defined inductively by
\beal\label{eq:defpolylog}
	\Li_{m}^{ \varphi}(z) \;:=\;  \frac 1 z \frac {d}{d z} \Li_{m+1}^{ \varphi}(z)
\eeal
where we note that $\Li^\varphi_{m}$ for $m\leq 0$ is independent of $\varphi$ and the branching of $\Li_1^\varphi$. Therefore, we will omit the index for $m\leq 0$ sometimes.

With this setup, we require some basic asymptotic estimates. 

\begin{lemma}\label{asymp_lemma} The following are true.  
\begin{enumerate}
 	\item Let $\alpha\in \R$. As $|t|\to\infty$ for $t$ on a ray in $\C$ we have with $\pm = \sign(\Re(t))$ 
			\bea
				\sin\left(\alpha (it-n_0)\right)
				\= \pm \frac{1}{2i} \exp\left(\pm \alpha(t -n_0i)\right)(1+{o}(|t|^{-L}))\\
			\eea
			for all $L\in\N$.\medskip
	\item Let $\zeta = e(\frac r m)$ be a root of unity of order $m\in\N$. As $z\to0$ in the right half-plane, i.e. $q=\zeta e^{-z}\to \zeta$, we have
\bea
(q;q)_\infty
\=
\exp\left(-\frac{\pi^2}{6m^2z} + \frac z{24}\right)
\sqrt{\frac{2\pi}{mz}} e\left(\frac{s(-r,m)}{2}\right) (1+o(|z|^L))
\eea
for all $L\in\N$. Here, $s(r,m)$ is the Dedekind sum defined by
		\bea
			s(r,m) \;:=\; \sum_{l=1}^{m-1}\frac l m \left(\frac{rl}m - \left\lfloor{\frac{rl}{m}}\right\rfloor-\frac 1 2\right).
		\eea
		\item Assuming the notation above, if $m$ is even we have
		\bea
			(-q;q)_\infty
			\= e^{-\frac{\pi^2}{6m^2z}+\frac z {12}}
			Q(\zeta) (1+o(|z|^L))
		\eea
		as $z\to 0$ in the right half-plane for all $L\in\N$ where
		\beal\label{eq:Qdef}
			Q(\zeta) \= e\left(\frac{s(-r,\tfrac m2)-s(-r,m)}{2}\right).
		\eeal
\end{enumerate}\end{lemma}

\begin{proof}  We prove each part of the lemma separately as follows.
\begin{enumerate}
\item
	We have
	\bea
		\sin(\alpha (it-n_0))
		\= \frac {e^{-\alpha (t+in_0)}-e^{\alpha (t+ in_0)}} {2i}
		\= \begin{cases}
			\frac{\exp\left({\alpha(t+in_0)}\right)}{2i}\left(1-e^{-\alpha(2t+2in_0)}\right),
			&\text{if } \Re(t)>0,\\
			-\frac{\exp\left({-\alpha(t+in_0)}\right)}{2i}\left(1-e^{\alpha(2t+2in_0)}\right),
			&\text{if }\Re(t)<0. 
		\end{cases}
	\eea
	In each case the second exponential becomes exponentially small as $t\to\infty$. \medskip
\item 
	It is well known that the eta function
	\bea
		\eta(\tau) := q^{\frac{1}{24}} (q;q)_\infty
	\eea
	where $q=e^{2 \pi i \tau}$ and $\tau\in\Hh$ satisfies the modular transformation formula
	\bea
		\eta\left(\frac{a\tau +b}{c\tau+d}\right)
		\= e\left( \frac{a+d}{24c}-\frac{s(d,c)}2-\frac 1 8\right)
		(c\tau+d)^{\frac 1 2}
		\eta(\tau)
	\eea
for all $\begin{psmallmatrix}a&b\\c&d\end{psmallmatrix}\in \operatorname{SL}_2(\Z)$ with $c>0$.  Hence, for $\tfrac rm\in\Q$
 with $(r,m)=1$ we choose $a,b\in\Z$ such that $\begin{psmallmatrix}-a&-b\\ m&-r\end{psmallmatrix}\in \operatorname{SL}_2(\Z)$. Then we obtain for all $\tau\in\Hh$ with $q=e(\tau)$
	\bea
		(q;q)_\infty
		\= q^{-\frac{1}{24}}\eta(\tau)
		\= &q^{-\frac{1}{24}}
		(m\tau-r)^{-\frac 1 2}
		e\left(\frac{a+r}{24m}+\frac{s(-r,m)}2+\frac 1 8\right)
		\eta\left(\frac{a\tau+b}{-m\tau +r}\right).\\
	\eea
	
Setting $\tau = \tfrac rm -\frac{z}{2\pi i}$ we have $m\tau-r = -\frac{mz}{2\pi i}$ and
\bea
\frac{a\tau+b}{-m\tau +r}
\=
2\pi i \frac{ar+bm}{m^2z}
- \frac{a}{m}
\=
\frac{2\pi i}{m^2z}
-\frac a m.
\eea
Because $z$ is in the right half-plane we have $\tau\in\Hh$ and we obtain
\bea[]
&\eta\left(\frac{a\tau+b}{-m\tau +r}\right)
\= {\eta}
\left(
\frac{2\pi i}{m^2z}
-\frac a m
\right)\\[5pt]
\=
&e\left(
\frac 1 {24}
\left(
\frac{2\pi i}{m^2z}
-\frac a m
\right)
\right)
\prod_{n\geq 1}
\left(1-
e\left(
n\left(
\frac{2\pi i}{m^2z}
-\frac a m
\right)
\right)\right)
\\[5pt]
\=&\exp\left(-\frac {\pi^2} {6m^2 z}\right)
e\left(\frac{-a}{24m}\right)
\prod_{n\geq 1}
\left(1-
\exp
\left(\frac{-4\pi^2n}{m^2z}
-\frac {2\pi ian} m
\right)\right).
\eea
As $z\to 0$ in the right half-plane we have
\bea
\prod_{n\geq 1}
\left(1-
\exp
\left(\frac{-4\pi^2n}{m^2z}
-\frac {2\pi ian} m
\right)\right)
\=  (1+ o(|z|^L))
\eea
for all $L\in\N$.
Hence, we obtain with $q=e^{2\pi i \tau} = e(\frac r m) e^{-z}$
\bea
(q;q)_\infty
\=
\exp\left({-\frac{\pi^2}{6m^2z}}+\frac z {24}\right) \sqrt{\frac{2\pi}{mz}} e\left(\frac{s(-r,m)}{2}\right) (1+ o(|z|^L))
\eea as claimed. \medskip

\item The claim follows from the previous statement and the identity $(-q;q)_\infty = \frac{(q^2;q^2)_\infty}{(q;q)_\infty}.$ \qedhere
\end{enumerate}
\end{proof}

Moreover, we will use a refinement of Lemma 2.1 in \cite{GZ} following Lemma 4 in  \cite{CWthesis}.

\begin{lemma}\label{lma:asymp_lemmaqpoch}
Let $w = e^{iv}\in\C$ such that if $\re(v)>0$ then $\re((v+2\pi n)/\varphi) \neq 0$ for all $n\in\Z$. Moreover, let $\zeta\in\C$ be a root of unity of order  $m\in\N$. Then we have { as $z\to 0$ in the right half-plane, i.e.} $q= \zeta e^{-z/m}\to \zeta$
		\bea
			(wq;q)_\infty \=
			\exp\left(-\frac{\Li_2^{\varphi}(w^m)}{mz}
			-\frac 1 2 \Li_1^\varphi(w^m)
			+\sum_{t=1}^m \frac t m \Li_1^{\varphi}(\zeta^{t}w)
			+\psi_{\zeta}(z;w)\right)
		\eea
	where $\psi_{\zeta}(z;w) \in \C[[z]]$ has an asymptotic expansion as $z\to0$
	\bea
		\psi_{\zeta}(z;w) =
		-\sum_{s=2}^{N}\sum_{t=1}^{m}  B_s\left(1-\frac t m\right)
		\Li_{2-s}(\zeta^{t} w) \frac{z^{s-1}}{s!}
		 +O(|z|^N)
	\eea
	for all $N\in\N$.
\end{lemma}
\begin{proof}
{ Throughout the proof we write $z=\varphi h$ where $h\in\R_{>0}$ and $\varphi\in\C$ with $|\varphi|=1$, $|\arg(\varphi)|<\frac \pi 2$.}
We have
\bea
	\widetilde \log(wq;q)_\infty
	\= \sum_{n\geq 1} \widetilde\log(1-wq^{n})
	\= \sum_{t=0}^{m-1} \sum_{k\geq 1}
		\widetilde\log\left(1-\zeta^{-t}we^{-\varphi h(km-t)/m}\right)
\eea
and apply the Euler-Maclaurin summation formula \cite[p.13]{ZMellin} to obtain for every $N\in\N$
\beal\label{eq:EMpf}
	\widetilde\log(wq;q)_\infty
	\= \sum_{t=0}^{m-1} &\frac 1 h \int_{0}^\infty
	\widetilde \log\left(1-\zeta^{-t}we^{-\varphi (x-th/m)}\right) dx\\
	&+\sum_{n=0}^N \frac{(-1)^n B_{n+1}}{(n+1)!}
	\frac {d^n}{dx^n}
	\widetilde \log\left(1-\zeta^{-t}we^{-\varphi (x-th/m)}\right) dx|_{x=0}
	\;h^n
	+\mathcal E_{t,N}
\eeal
with
\bea
	\mathcal E_{t,N}
		\coloneqq	h^N\int_0^\infty
		\Li_{1-N}\left(\zeta^{-t}we^{-\varphi (x-th/m)}\right)
			\frac{\overline B_N(x)}{N!}dx.
\eea
The Euler-Maclaurin summation formula applies in this case, as the function defined by $x\mapsto 	\widetilde \log\left(1-\zeta^{-t}we^{-\varphi (x-th/m)}\right)$ and all of its derivatives
\beal\label{eq:EMderivs}
	\frac {d^n} {dx^n}
	\widetilde\log\left(1-\zeta^{-t}we^{-\varphi (x -th/m)}\right)
	\=
	-(-1)^n \varphi^n \Li_{1-n}^\varphi(\zeta^{-t}we^{-\varphi x+ zt/m})
\eeal
are of rapid decay as $x\to \infty$.

We have
\bea
	\frac 1 h
	\int_{0}^\infty
	\widetilde\log\left(1-\zeta^{-t}we^{-\varphi (x-th/m)}\right)
	dx
	\= &-\frac 1 z
	\Li_2^\varphi({e^{zt/m}} \zeta^{-t}w)\\
	\= &-\frac 1 z 
	\sum_{l\geq 0}
	\frac {\Li_{2-l}^\varphi(\zeta^{-t}w)}{l!}
	\left(\frac {zt}m\right)^l
\eea
where $\widetilde \Li_s^\varphi$ are the polylogarithms defined in (\ref{eq:defpolylog}).
Evaluating the derivatives (\ref{eq:EMderivs}) at $x=0$ gives
\bea
		\frac {d^n} {dx^n}
	\widetilde\log\left(1-\zeta^{-t}we^{-\varphi (x-th/m)}\right)|_{x=0}
	\=&-(-1)^n \varphi^n \Li^{\varphi}_{1-n}(\zeta^{-t}we^{zt/m})\\
	\=&-(-1)^n \varphi^n \sum_{k\geq 0}\frac{\Li_{1-n-k}^{\varphi}(\zeta^{-t}w)}{k!} \left(\frac {zt}m\right)^k.
\eea Hence, (\ref{eq:EMpf}) can be written as
\bea
	\log(wq;q)_\infty
	\= & - \sum_{t=0}^{m-1}\frac 1 z 
	\sum_{l\geq 0}
	\frac {\Li_{2-l}^\varphi(\zeta^{-t}w)}{l!}
	\left(\frac {zt}m\right)^m\\
	&-\sum_{t=0}^{m-1}\sum_{n=0}^N \frac{B_{n+1}}{(n+1)!}
	\sum_{k\geq 0}\frac{\Li_{1-n-k}^\varphi(\zeta^{-t}w)}{k!} \left(\frac {t}m\right)^k z^{k+n} +\mathcal E_{t,N}
\eea
	by using the distribution property of the dilogarithm (\cite[p.9]{Zdilog}). Moreover, shifting $n\mapsto n-1$ and summing over $s=n+k=1\ldots N+1$ shows that $\log(wq;q)_\infty$ is equal to
\bea[]
	& -\frac{\Li_2^\varphi(w^m)}{mz}
	-\sum_{t=0}^{m-1}\sum_{s=1}^{N+1}
	\sum_{n=0}^s \binom s n B_n
	\left(\frac tm\right)^{s-n}
	\Li_{2-s}^\varphi(\zeta^{-t}w)\frac{z^{s-1}}{s!}
	+\mathcal E_{t,N}+O(|z|^N)\\
	\=&
	-\frac{\Li_2^\varphi(w^m)}{mz}
	-\sum_{t=0}^{m-1} \Li_1^\varphi(\zeta^{-t}w)\left(\frac t m - \frac 1 2\right)
	-\sum_{t=0}^{m-1}\sum_{s\geq2} B_s\left(\frac t m\right)
	\Li_{2-s}(\zeta^{-t}w) \frac{z^{s-1}}{s!}
	+\mathcal E_{t,N}+O(|z|^N)
\eea
 where we collect all terms with $s\geq N$ in $O(|z|^N)$.
Replacing $t$ by $m-t\in\{1,\ldots,m\}$ yields
\bea[]
	&-\frac{\Li_2^\varphi(w^m)}{mz}
	+\sum_{t=1}^{m} \Li_1^\varphi(\zeta^{t}w)\left(\frac t m - \frac 1 2\right)
	-\sum_{t=0}^{m-1}\sum_{s\geq2} B_s\left(1-\frac t m\right)
	\Li_{2-s}(\zeta^{t}w) \frac{z^{s-1}}{s!}
	 + \mathcal E_{t,N}+O(|z|^N)\\	
	\=&-\frac{\Li_2^\varphi(w^m)}{mz}
	-\frac 1 2 \Li_1^\varphi(w^m)
	+\sum_{t=0}^{m-1} \frac t m \Li_1^\varphi(\zeta^{	t}w)
	+\psi_{w,\zeta}(z),
\eea
where we used
\bea
	\sum_{t=1}^m \Li_1^\varphi(\zeta^tw)
	\=\sum_{t=1}^m \widetilde\log(1-\zeta^tw)
	\= \widetilde\log\left(\prod_{t=1}^m 1-\zeta^tw\right)
	\= \widetilde \log( 1-w^m)
	\= \Li_1^\varphi(w^m).
\eea
Note that $\Li_{1-N}(z) \in (1-z)^{N} \C[z]$ for $N>0$. Our assumption implies that $we^{-\varphi xh}\neq1$, hence there exists $C>0$ such that $|\Li_{1-N}(we^{-\varphi (hx+th/m)})| < C |we^{-\varphi (hx+th/m)}| < C |we^{-\varphi xh)}|$. Hence, we obtain for some $D>0$
\bea
	|\mathcal E_{N,t}| \;\leq\;
	&\frac{h^N}{N!}C \int_0^\infty |we^{-\varphi  xh}| \overline B_N(x)dx\\
	\=&\frac{h^N}{N!} C \sum_{x_0=0}^\infty
	\int_0^\infty |we^{-\varphi h (x_0+x)}| B_N(x)dx\\
	\=&\frac{h^N}{N!} C|w| \sum_{x_0=0}^\infty
	|e^{-\varphi h x_0}|
	\int_0^1 |e^{-\varphi h x}| B_N(x)dx\\
	\=&\frac{h^N}{N!} \frac{C|w| }{1-|e^{-\varphi h}|} D 
	\= O (h^N),
\eea
since $|e^{-\varphi h}|<1$.
In particular,  $\psi_{\zeta}(z;w)$ has the claimed asymptotic expansion. This completes the proof.
\end{proof}

\section{Proof of Theorem \ref{Thm: radial asymp}}\label{Sec: proof of conj 1}

\subsection{Proof of Theorem \ref{Thm: radial asymp} (1)}
We begin by showing that at any root of unity with order not divisible by $4$, $v_1(q)$ converges.
\begin{lemma}\label{Lem: bounds at roots of unity not divis by 4}
	 Let $\zeta_N := e^{\frac{2 \pi i}{N}}$.  For any root of unity $\zeta_m^\ell$ with  $\gcd(\ell,m)=1$ and $4\nmid m$, we have that 
	$$
	\lim_{z\to 0}
	v_1(\zeta_m^\ell e^{-z}) \= 
	v_1(\zeta_m^\ell) \= 
 \sum_{s=0}^{m-1}  \frac{\zeta_{2 m}^{\ell s(s+ 1)}}{(-\zeta_m^{2\ell};\zeta_m^{2\ell})_{s}}
		\times 
		\begin{cases}
		2 &\text{if $m$ is odd,}\\ 
		\frac 4 5 &\text{if $m\equiv 2\bmod 4$}\\ 
		\end{cases}$$
	as $z\to 0$ on a ray in the right half-plane.
\end{lemma}
\begin{proof}
Let $z =he^{i\vartheta}$ for $\vartheta\in (-\frac \pi 2,\frac \pi 2)$.
With $q = \zeta_m^\ell e^{-z}$, we have
\bea
\frac{q^{n(n+1)/2}}{(-q^2;q^2)_n}
\= \prod_{j=1}^n \frac{1}{q^j+q^{-j}}
\= \frac 1{2^n}\prod_{j=1}^n \frac{1}{\cosh{(2\pi i j\ell/m+jz)}}.
\eea
For fixed $\vartheta\in (-\frac \pi 2,\frac \pi 2)$,
 $\cosh{(2\pi i j\ell/m+xe^{i\theta})}$ only depends on $\ell j \bmod m$ and never vanishes for $x\geq 0$ because $4\nmid m$. Moreover, for each $j=0,\ldots,m-1$, we have 
$\log|\cosh(2\pi i j\ell/m+xe^{i\vartheta})|\to \log|\cos(2\pi j\ell/m)|$ as $x\to 0$ and 
$\log|\cosh(2\pi i j\ell/m+xe^{i\vartheta})|\sim A_{j}x$ as $x\to \infty$ for some $A_j>0$.  Using this, as well the fact that 
 there are only finitely many congruence classes modulo $m$,  a calculus argument (e.g., using the Taylor expansion around $x=0$) reveals that there exist $a,b >0$ that only depend on $\vartheta$ and $m$ such that
\bea
\log|\cosh(2\pi i j\ell/m+xe^{i\vartheta})|
>ax-b\sqrt{x} +\log|\cos(2\pi j\ell/m)|
\eea  for all $x\geq 0$,   $j\in \mathbb Z_{\geq 0}$.
Thus, we have that
\bea
\left|\frac{q^{n(n+1)/2}}{(-q^2;q^2)_n}\right|
<
&\frac 1 {2^n}
\prod_{j=1}^n
\frac{\exp(-ajh+b\sqrt{jh})}
{|\cos(2\pi j\ell/m)|}\\
\leq
&\frac 1 {2^n}
{\exp(-an(n+1)h/2+{{bn^{3/2}\sqrt{h}}})}
\prod_{j=1}^n
\frac 1 
{|\cos(2\pi j\ell/m)|},
\eea
where we have used the inequality $\sum_{j=1}^n \sqrt{j} \leq {{ n^{3/2}}}$.
There exists a constant $C>0$ such that
\bea
{\exp(-an(n+1)h/2+{{bn^{3/2}\sqrt{h}}})} < C
\eea
for all $h\geq 0$ and $n\in\Z_{\geq0}$.
Because $4\nmid m$, we have
$
\prod_{j=1}^m {|\cos(2\pi j\ell/m)|} =  
2^{\rho_m-m}$,  where $\rho_m = 1$ if $m$ is odd, and $\rho_m=2$ if $m\equiv 2 \pmod{4}$. 
Thus, we see that
\bea
|v_1(\zeta_m^\ell e^{-z})|
\leq
&\sum_{n=0}^\infty
\left|\frac{q^{n(n+1)/2}}{(-q^2;q^2)_n}\right|
\\
\leq &\;C
\sum_{n=0}^\infty \frac 1 {2^n}
\prod_{j=1}^n
\frac 1 
{|\cos(2\pi j\ell/m)|}\\
\=
&C\sum_{s=0}^{m-1}
\sum_{n=0}^\infty \frac 1 {2^{s+mn}}
\prod_{j=1}^{s+mn}
\frac 1 
{|\cos(2\pi j\ell/m)|}
&\qquad{(n \mapsto s+mn)}
\\
\=
&C\sum_{s=0}^{m-1}
\sum_{n=0}^\infty \frac 1 {2^{s+mn}}
\prod_{l=0}^{n}
\prod_{k=1}^{m}
\frac 1 {|\cos(2\pi k\ell/m)|}
\prod_{j=1}^s
\frac 1 {|\cos(2\pi j\ell/m)|}
&\qquad{(j \mapsto l m+k)}
\\
\=
&C\sum_{s=0}^{m-1}
\prod_{j=1}^s
\frac 1 {2^s|\cos(2\pi j\ell/m)|}
\sum_{n=0}^\infty {2^{(m- \rho_m )(n+1)-mn}}\\
\=
& \frac{C}{2^{\rho_m}} \sum_{s=0}^{m-1}
\prod_{j=1}^s
\frac {2^m} {2^s|\cos(2\pi j\ell/m)|}
\sum_{n=0}^\infty \frac 1{2^{ \rho_m n}}\\
\=
&  \frac{C}{2^{\rho_m}(1-2^{-\rho_m})}\sum_{s=0}^{m-1}
\prod_{j=1}^s
\frac {2^m} {2^s|\cos(2\pi j\ell/m)|}
\eea is finite.
Hence, the theorem of dominated convergence implies that we can interchange the summation over $n$ and the limit $q\to \zeta_{m}^\ell$. In other words, we have
\bea
\lim_{h\to 0}v_1(\zeta_m^\ell e^{-h}) = v_1(\zeta_m^\ell).
\eea
More precisely, with the notation from above we have
	\begin{align*}
		v_1(\zeta_m^\ell) &\= \sum_{n=0}^\infty \frac{\zeta_{2 m}^{\ell n (n+1)}}{(-\zeta_m^{2\ell};\zeta_m^{2\ell})_n}
		\= \sum_{s=0}^{m-1} \sum_{n=0}^\infty \frac{\zeta_{2 m}^{\ell (s+mn) (s+mn+1)}}{(-\zeta_m^{2\ell};\zeta_m^{2\ell})_{s+mn}}  \\
		&\= \sum_{s=0}^{m-1} \sum_{n=0}^\infty
		\frac{(-1)^{(m+1)n}}{2^{\rho_m n}}
		\frac{ \zeta_{2 m}^{\ell s(s+ 1)}}{(-\zeta_m^{2\ell};\zeta_m^{2\ell})_{s}} 
		\= \left( \sum_{n=0}^\infty \frac{(-1)^{(m+1)n}}{2^{\rho_m n}}\right) \left(  \sum_{s=0}^{m-1}  \frac{\zeta_{2 m}^{\ell s(s+ 1)}}{(-\zeta_m^{2\ell};\zeta_m^{2\ell})_{s}} \right) \\
		&\=  \sum_{s=0}^{m-1}  \frac{\zeta_{2 m}^{\ell s(s+ 1)}}{(-\zeta_m^{2\ell};\zeta_m^{2\ell})_{s}}
		\times 
		\begin{cases}
		2 &\text{if $m$ is odd,}\\ 
		\frac 4 5 &\text{if $m\equiv 2\bmod 4$.}\\ 
		\end{cases}
	\end{align*} 
\end{proof}

\subsection{Proof of Theorem \ref{Thm: radial asymp}, (3)}\ Throughout we assume that $\varphi \in \C$ with $|\varphi|=1$ and $0\neq|\arg(\varphi)| < \frac \pi 2$ and write $z=\varphi h$.
We will present the case $q = ie^{-z}\to i$ for $z=\varphi h\to0$ in a fixed ray in the right half-plane in detail. The case  $q\to -i$ is analogous, and so we omit the proof for brevity.
\\

We split up the sum {defining $v_1(q)$} depending on $n\pmod 2$, i.e., consider separately
\beal\label{eq:defv10v11}
	v_1^{[0]}(q) \; = \sum_{n\geq 0 \text{ even}} \frac {q^{n(n+1)/2}}{(-q^2;q^2)_n} \=
	\frac 1 {(-q^2;q^2)_\infty} \sum_{n\geq 0 \text{ even}} (-i)^{n/2} e^{-zn(n+1)/2}\ (-e^{-2nz}q^2;q^2)_\infty,\\
	v_1^{[1]}(q) \;= \sum_{n\geq 0 \text{ odd}} \frac {q^{n(n+1)/2}}{(-q^2;q^2)_n} \= 
	\frac 1 {(-q^2;q^2)_\infty} \sum_{n\geq 0 \text{ odd}} i^{(n+1)/2} e^{-zn(n+1)/2}\ (e^{-2nz}q^2;q^2)_\infty,\\
\eeal

as we have 
\begin{equation*}
	i^{n(n+1)/2} \= \begin{cases}
		(-i)^{n/2},	&\text{if $n$ is even,}\\
		i^{(n+1)/2}, &\text{if $n$ is odd.}
	\end{cases}
\end{equation*}

To state our next result, recall that for $n \geq 0$ the $q$-Pochhammer symbol satisfies the classical formula
\begin{align*}
	(a;q)_n = \frac{(a;q)_\infty}{(aq^n;q)_\infty},
\end{align*}
meaning that we may extend the definition of the $q$-Pochhammer symbol  to $-n$ by defining
\begin{align*}
	(a;q)_{-n} \coloneqq
\frac{(a;q)_\infty}{(aq^{-n};q)_\infty}
=
\frac{1}{(aq^{-n} ; q)_n}.
\end{align*}
Then the $q$-Pochhammer symbol satisfies  
\bea
(a;q)_{-n}
\= q^{n(n-1)/2}
\frac{(-q/a)^n}{(q/a;q)_n}.
\eea
We will prove the following proposition which implies (3) in Theorem \ref{Thm: radial asymp}.
\begin{prp} As  $z \to0$ in the right half-plane on a ray with $\arg z \neq 0$, we have with $q=ie^{-z}$
	\begin{align}\label{final_asymp_v10}
	v_1^{[0]}(q)
	\= 
	&e^{-\frac{V}{z} }
\sqrt{\frac{2\pi i}{-z}}\; \gamma^-_{(1/4)}{(z)}(1+ O(|z|^{L}))
	+\phi^{[0]}_{(1/4)}(z){(1+O(|z|^L))},
\\\label{final_asymp_v11}
	v_1^{[1]}(q)
	\=
	 &e^{ \frac{V}{z} }  \sqrt{\frac{2\pi i}{z}}\; \gamma^+_{(1/4)}{(z)}(1+ O(|z|^{L}))
	+\phi^{[1]}_{(1/4)}(z)
{(1+O(|z|^L))}
\end{align}
{for all $L>0$,}
where $\gamma^{\pm}_{(1/4)}(z)\in\C[[z]]$ is defined in 
 
{\eqref{eq:defgammam}, resp. \eqref{eq:defgammam2}}
and  
{$
\phi^{[0]}_{(1/4)}(z), 
\phi^{[1]}_{(1/4)}(z)
\in\C[[z]]$
}
with
\bea
	\phi^{[0]}_{(1/4)}(z)
		&\;= \sum_{\substack{n<0:\\ n\equiv 0\Mod{2}}} \frac{q^{n(n+1)/2}}{(-q^2;q^2)_n}
		\=	- 4 i \, z
			- 48 \, z^{2}
			+ \frac{2878}{3} i \, z^{3}
			+ 26704 \, z^{4}
			+O(|z|^5),\\
	\phi^{[1]}_{(1/4)}(z)
		&\;= \sum_{\substack{n<0:\\ n\equiv 1\Mod{2}}} \frac{q^{n(n+1)/2}}{(-q^2;q^2)_n}
		\= 2 + 8 i\, z - 96 \, z^{2}
		 - \frac{5708}{3} i \, z^{3}
		 + 52640 \, z^{4} 
		+O(|z|^5).
\eea
\end{prp}

\subsubsection{Watson's contour integral} Using Watson's contour integral (\cite[4.2]{GR}, \cite{Watson1910}), we establish the following integral representation for $v_1^{[0]}(q)$ and $v_1^{[1]}(q)$.
\begin{lemma}\label{int_repv1}
	For $q=ie^{-z}$ with $\re(z)>0$, we have
	\begin{align}\label{eq:int_repv10}
		v_1^{[0]}(q)
		&\= \frac {-1}{2 i} \frac 1 {(-q^2;q^2)_\infty}
		\int_{L_\infty}
			e^{\pi i s/4}
			e^{-zs(s+1)/2}
			(-e^{-2sz}q^2;q^2)_\infty
			\frac {1} {2\sin\left({\pi s}/{2}\right)}
		ds,\\[10pt]
	\label{eq:int_repv11}
	v_1^{[1]}(q)
	&\= \frac {e^{\pi i 3/4}}{2} \frac 1 {(-q^2;q^2)_\infty}
		\int_{L_\infty}
			e^{-\pi i s/4}
			e^{-zs(s+1)/2}
			(e^{-2sz}q^2;q^2)_\infty
			\frac {1} {2\cos\left({\pi s}/{2}\right)}
		ds
	\end{align}
	where $L_\infty$ is the contour depicted
	in Figure \ref{figure_contour} as $R \to \infty$ for some small $\varepsilon>0$.
\end{lemma}
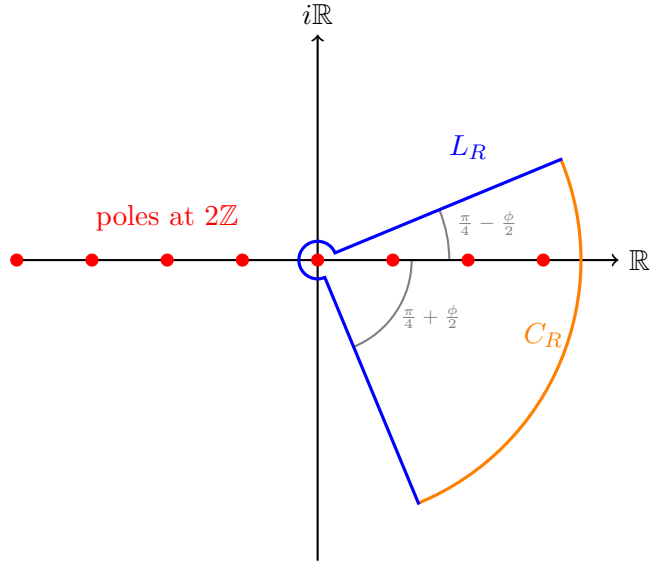
\begin{figure}[h]
\centering 
\input{contour2}
\caption{Contours $L_R$ and $C_R$}\label{figure_contour}
\end{figure}

\begin{proof}
We will prove the statement for $v_1^{[0]}(q)$ in detail. The proof for $v_1^{[1]}(q)$ follows analogously.\\
As mentioned above, we write $z = \varphi h \in\C$ where $h\in\R_{>0}$ and  $\varphi\in\C$ with $|\varphi|=1$ and $|\arg\varphi |<\frac \pi 2$.
The function
$ \frac{1}{\sin\left({\pi s}/{2}\right)}$
has poles at $s\in 2\Z$ with residues
$(-1)^{s/2}\; \frac 2\pi$.

Hence, with $(-i)^{s/2} (-1)^{s/2} = i^{s/2} = e^{\pi i s/4}$ for $s\in2\Z$, we obtain by Cauchy's theorem - if the subsequent integrals are convergent  - using the contours from Figure \ref{figure_contour}
\bea
	&\frac {-1}{2 i} \frac 1 {(-q^2;q^2)_\infty}
	\lim_{R\to \infty}
		\int_{L_R+C_R}
			e^{\pi i s/4}
			e^{-zs(s+1)/2}
			(-e^{-2sz}q^2;q^2)_\infty
			\frac {1} {2\sin\left({\pi s}/2\right)}
		ds\\
	\= &\frac 1 {(-q^2;q^2)_\infty} \sum_{n\geq 0 \text{ even}}
	\mathop{\operatorfont{Res}}_{s=n}\left(
		(-i)^{s/2}
		e^{-zs(s+1)/2}
		(-e^{-2sz}q^2;q^2)_\infty
		\frac {\pi(-1)^{s/2}} {2\sin\left({\pi s}/2\right)}
	\right)\\
	\= &\frac 1 {(-q^2;q^2)_\infty} \sum_{n\geq 0 \text{ even}} (-i)^{n/2} e^{-zn(n+1)/2}\ (-e^{-2nz}q^2;q^2)_\infty\\
	\= 	&v_1^{[0]}(q).
\eea

It remains to prove the following 2 claims.
\begin{enumerate}
	\item The integral over $L_\infty$ converges.
	\item The integral over the arc $C_R$ vanishes as $R\to\infty$.
\end{enumerate}
Before proving (1) and (2) we make some initial observations.
If we parameterize $L_R$ and $C_R$ { away from the indention around $0$} by $s=re^{i \theta}$ with $0<r\leq R$ and
$\theta \in {\left[-\frac \pi 4 - \frac{{\arg\varphi}} 2+{\varepsilon}, \frac \pi 4 - \frac{{\arg\varphi}} 2 - \varepsilon\right]}$. Then we have $-zs^2 = -hr^2e^{i({\arg\varphi}+2\theta)}$ with
{${\arg\varphi}+2\theta\in[-\frac \pi 2 + 2\varepsilon, \frac \pi 2 - 2\varepsilon]$ for $\varepsilon>0$}, i.e. $\re(-zs^2)<0$. Similarly, one checks $\re(-zs)<0$.

In particular, the Pochhammer symbol can be uniformly bounded by 
\beal\label{bound_pochhammer}
	\left|(-e^{-2sz}q^2;q^2)_\infty\right|
	\;\leq\;  \prod_{j\geq 1} 1+ |e^{-2sz}||q^{2j}|
	\=\prod_{j\geq 1} 1+ |e^{-2\re(sz)}||e^{-2j\re(z)}|
	\;<\;\prod_{j\geq 1} 1+ |e^{-2j\re(z)}|,
\eeal
since $-\re(sz)<0$. 
Hence, we have 
\beal\label{first_int_est}
	&\left|\int_{{L_R'+} C_R}
		e^{\pi i s/4}
		e^{-zs(s+1)/2}
		(-e^{-2sz}q^2;q^2)_\infty
		\frac {1} {2\sin\left({\pi s}/2\right)}\right|\\
	&\hspace{.2in}\leq
	\prod_{j\geq1} (1+ |e^{-2j\re(z)}|)
	\int_{{L_R'+} C_R}
		e^{\re(\pi i s/4-zs(s+1)/2)}
		\left|\frac {1} {2\sin\left({\pi s}/2\right)}\right|
	\;ds 
\eeal
{where $L_R'$ denotes the part of $L_R$ away from the indention around $0$.}

As ${\im}|s| \to \infty,$ we have for all $L\in\N$
\begin{equation*}
	|\sin(\pi s/2)|
{	\= \frac 1 2 e^{\pi |\im(s)|/2} (1+o(|s|^{{-L}})}
\end{equation*}
and with $s=re^{i\theta}$ we compute
{
\bea
	\;&\hspace{-.6in}\re\left(\frac{\pi i s}4-\frac{zs(s+1)}2\right)-\frac{\pi |\im(s)|}2\\
	\=&{
	\re\left(
	  \frac{\pi i re^{i\theta}}4
	  - \frac{he^{i{\arg\varphi}}r^2e^{2i\theta}}2
	 -\frac{he^{i{\arg\varphi}}re^{i\theta}}2
	\right)
	- \frac{\pi |\im(s)|}2
	}\\
	\=&
	-\frac{\pi r \sin(\theta)}4+
	\re\left(
	  - \frac{hr^2e^{i({\arg\varphi}+2\theta)}}2
	 -\frac{hre^{i({\arg\varphi}+\theta)}}2
	\right)
	- \frac{\pi |\im(s)|}2
	\\
	\=&-\frac{\pi r  \sin (\theta)}4
	- \frac{r^2h^2\cos({\arg\varphi}+2\theta)}2
	-\frac{rh\cos({\arg\varphi}+\theta)}2
	-\frac{r\pi|\sin\theta|}2\\
	\=&
	-\frac{r^2h^2\cos({\arg\varphi}+2\theta)}2
	-r\left(\frac{\pi  \sin (\theta)}4
	+\frac{h\cos({\arg\varphi}+\theta)}2
	-\frac{\pi |\sin\theta|}2\right).
\eea}
{ Hence, the exponent in the integrand in (\ref{first_int_est}) is eventually negative}, since ${\arg\varphi}+2\theta\in (-\frac \pi 2, \frac \pi 2)$ and thus $\cos({\arg\varphi}+2\theta)>\delta>0$ for some $\delta$.\\
More precisely, we have for some constant $M>0$, {uniformly in $\theta$},
\bea
	\;&\hspace{-.7in}
	-\frac{r^2h^2\cos({\arg\varphi}+2\theta)}2
	-r\left(
	\frac{\pi  \sin (\theta)}4
	+\frac{h\cos({\arg\varphi}+\theta)}2
	-\frac{\pi |\sin\theta|}2\right)\\
	<&\;{-\frac{r^2h^2 \delta}2
	+ r\left(\frac \pi 4+\frac h 2-\frac \pi 2\right)}
	\;<\;-Mr^2
\eea
for $R$ and $r$ large enough. 

Therefore, it is sufficient to prove both claims for the integral
\bea
	\int
		e^{-Mr^2}
	\;ds.
\eea\\
\textit{Claim (1): The integral over $L_\infty$ converges.}\newline
We consider the integral along the contour $\{re^{i\theta_\pm}, r\in\R_{>0}\}$ with $\theta_\pm = \pm \frac \pi 4 - \frac{\arg\varphi}2$. By the discussion above, the integral is
\bea
	O\left(\int_{0}^{R}
		e^{-Mr^2}
	\;dr\right),
\eea
which converges as $R\to \infty$.
\smallskip \ \\
\textit{Claim (2): The integral over the arc $C_R$ vanishes as $R\to\infty$.}\newline
Similarly, we see that the integral over $C_R$ is eventually bounded by a constant times
\bea
	\int_{\left(-\frac \pi 4 - \frac{{\arg\varphi}} 2, \frac \pi 4 - \frac{{\arg\varphi}} 2\right)}
		e^{-MR^2}
	\;d\theta \to 0
\eea
as $R\to \infty$.
\end{proof}

\subsubsection{Sum over even integers}\label{sec:sumeveni}
We will prove the statement for $v_{1}^{[0]}$ in detail. As the proof for $v_{1}^{[1]}$ follows {\it mutatis mutandis}, we will only sketch it in Section \ref{sec:Sum over odd parts}.

\begin{proof}[Proof of (\ref{final_asymp_v10})]
We use the integral representation from Lemma \ref{int_repv1} with $q= ie^{-z}$ and substitute $s = iv/z$ to obtain
\begin{equation}\begin{aligned}\label{int_rep0}
	v_1^{[0]}(q)
	&\= \frac {-1} {2z(-q^2;q^2)_\infty}
		\int_{-izL_\infty}
			e^{-\pi v/4z}
			e^{v^2/2z-iv/2}
			(-e^{-2iv}q^2;q^2)_\infty
			\frac{1}{\sin\left(\frac{\pi i v}{2z}\right)}
		dv,
\end{aligned}\end{equation}
where the contour $-izL_\infty$ is depicted in Figure \ref{fig:contourihLinfty}.

\begin{figure}[h]
\centering 
\input{contourihLinfty}
\caption{The Contour $-i\varphi L_\infty$}\label{fig:contourihLinfty}
\end{figure}
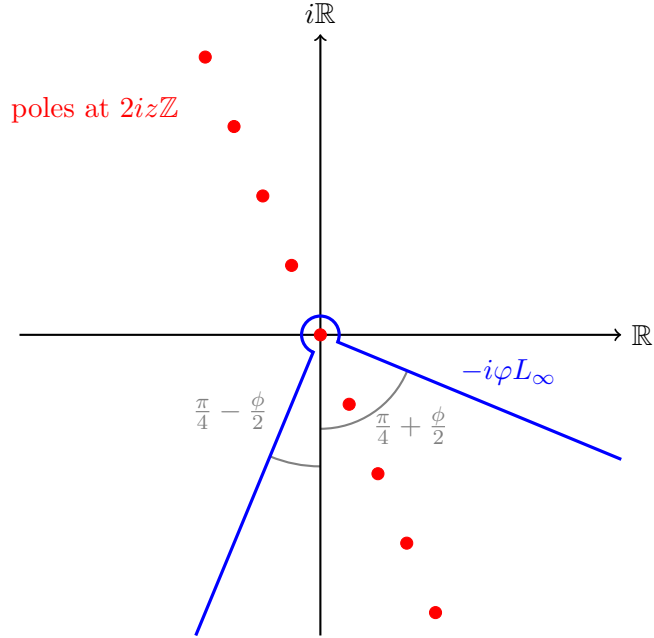

We consider the integral representation (\ref{int_rep0}) of $v_1^{[0]}$ and change the contour $-izL_\infty$ to a contour $\mathcal S$ with fixed minimum distance from $0$ and passing through $-\frac \pi {12}$.
The poles of the integrand in (\ref{int_rep0}) are at $v\in i2z\Z$, and as $z\to 0$ they accumulate at $0$.
Hence, if we integrate along the contour $\mathcal S$, all poles at $i2z\Z_{<0}$ eventually get shifted to the other side of the contour (cf. Figure \ref{fig:contourreimS}). In other words the integral \eqref{int_rep0} can be written as
	\beal\label{int_rep}
	\frac {-1} {2z(-q^2;q^2)_\infty}
	&\int_{\mathcal S}
			e^{-\pi v/4z}
			e^{v^2/2z-iv/2}
			(-e^{-2iv}q^2;q^2)_\infty
			\frac {1} {\sin\left(\frac{\pi i v}{2z}\right)}
		dv\\
		&\;+
		\frac {-1} {2z(-q^2;q^2)_\infty}
		\sum_{\substack{n<0:\\|2zn|<d_0}}
		  \mathop{\operatorname{Res}}_{v=-2izn}
		    {\left(
		   		e^{-\pi v/4z}
				e^{v^2/2z-iv/2}
				(-e^{-2iv}q^2;q^2)_\infty
				\frac {1} {\sin\left(\frac{\pi i v}{2z}\right)}
		    \right)}
	\eeal for some $d_0>0$, which is the distance from $0$ to where the contour crosses the line of poles.
{The residue of
$\frac 1{\sin\left(\frac{\pi i v}{2z}\right)}$
at $v=-2izn$ is given by $-(-1)^n 2z$ and thus
\bea[]
&\mathop{\operatorname{Res}}_{v=-2izn}
{\left(
e^{-\pi v/4z}
e^{v^2/2z-iv/2}
(-e^{-2iv}q^2;q^2)_\infty
\frac {1} {\sin\left(\frac{\pi i v}{2z}\right)}
\right)}
\=
(-i)^n e^{-(2n^2-n)z}
(-e^{-4nz}q^2;q^2)_\infty.
\eea
}
 
As $z\to 0$ the residues can be collected in
	\bea
		\phi^{[0]}_{(1/4)}(z)
		&\;:=\;\frac {-1} {2z(-q^2;q^2)_\infty}
		\sum_{\substack{n<0}}
		 \mathop{\operatorname{Res}}_{v=-2izn}
		    {\left(
		   		e^{-\pi v/4z}
				e^{v^2/2z-iv/2}
				(-e^{-2iv}q^2;q^2)_\infty
				\frac {1} {\sin\left(\frac{\pi i v}{2z}\right)}
		    \right)}\\[7pt]
&\=
\frac{1}{(-q^2;q^2)_\infty}
\sum_{n<0}
(-i)^n
e^{-(2n^2-n)z}
(-e^{-4nz}q^2;q^2)_\infty.
\eea
Using the $q$-Pochhammer symbol for negative indices this implies with $q = ie^{-z}$
\bea
\phi^{[0]}_{(1/4)}(z)
&\= \sum_{\substack{n<0:\\ n\equiv0\bmod 2}} \frac{q^{n(n+1)/2}}{(-q^2;q^2)_n}
		\=\sum_{\substack{l>0:\\ l\equiv0\bmod 2}}
			q^{-l(l-1)/2}(-1;q^2)_l
		\=2\!\!\!\sum_{\substack{m>0:\\ m\equiv1\bmod 2}}\!\!
			q^{-m(m+1)/2}(-q^2;q^2)_m.
\eea
We have
\bea
(-q^2;q^2)_m = \prod_{j=1}^m \bigl(1+(-1)^je^{-2jz}\bigr)
\in z^{\lceil m/2\rceil}\C[[z]]
\eea
 since $1+(-1)^je^{-2jz}\in z\C[[z]]$ for $j =1,\ldots, m$ odd.
Hence, $\phi^{[0]}_{(1/4)}(z)\in\C[[z]]$
and the first terms are given by
\bea
\phi^{[0]}_{(1/4)}(z)
&\=-4\*i\*z
 - 48\*z^2
 + \frac{2878}{3}\*i\*z^3
 + 26704\*z^4
 - \frac{28574401}{30}\*i\*z^5
 - \frac{207245984}{5}\*z^6\\
 &\quad+ \frac{2683138049759}{1260}\*i\*z^7
 + \frac{13245720939344}{105}\*z^8
 - \frac{768005626895809921}{90720}\*i\*z^9
+O(z^{10}).
	\eea

\begin{figure}
\centering 
\input{contourS.tex}
\caption{A contour $\mathcal S$ (before applying Lemma \ref{asymp_lemma}).
}\label{fig:contourreimS}
\end{figure}
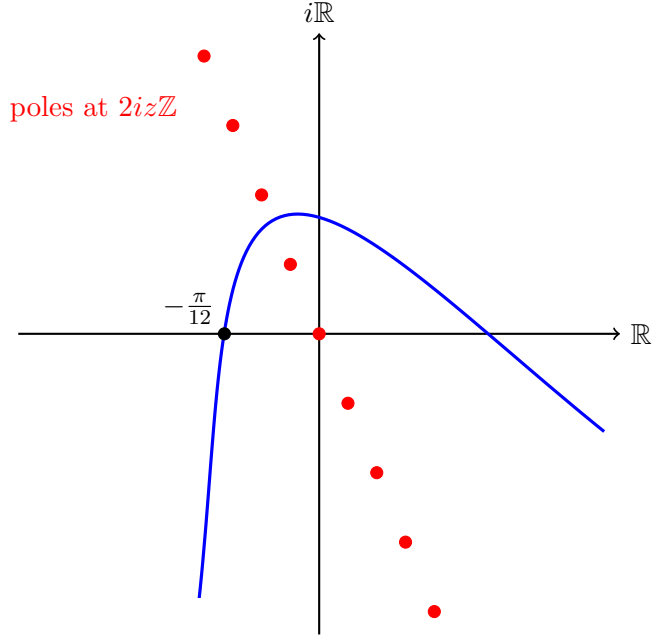

We apply Lemmas \ref{asymp_lemma} and \ref{lma:asymp_lemmaqpoch} to the integrands in \eqref{int_rep}, to obtain that $v_1^{[0]}(q)$ is {asymptotically} equal to
\begin{align}\label{int_rep3}
\frac {ie^{\pi^2/48z{-z/{12}}}}{2z}\!\! 
\int_{\mathcal S}
&\sign(\re(v/\varphi))
\exp\left(
\frac{
 -\Li_2^\varphi(e^{-4iv})+4v^2-2\pi v
 - 4\sign(\re(v/\varphi))\pi v}
 {8z}\right)\\
&
{\exp\left(-\frac{iv}2+ \frac {\Li_1^\varphi(-e^{-2iv})}2
+\psi_{-1}(4z;-e^{-2iv})
\right)
dv}
\;+\; 
\phi^{[0]}_{(1/4)}(z)
\end{align} 
which can be rewritten with
\begin{equation}\begin{aligned}\label{f_pm_def}
	f(v) &\= -\frac{\Li_2^\varphi(e^{-4iv})}8+\frac{v^2}2-\frac {\pi v}4- \sign(\re(v/\varphi)) \frac{\pi v}2,\\
{
g(z;v)} &\= 
{
\sign(\re(v/\varphi))
\exp\left(-\frac{iv}2+ \frac {\Li_1^\varphi(-e^{-2iv})}2
+\psi_{-1}(4z;-e^{-2iv})
\right),}
\end{aligned}\end{equation}
as
\beal\label{int_rep4}
v_1^{[0]}(q) &\=
		\frac {i}{2z} e^{\pi^2/48z{-z/{12}}}
		\int_{\mathcal S} 
			e^{f(v)/z} {g(z;v)}
		dv\;
{(1+O(|z|^L))}
		\;+\;\phi^{[0]}_{(1/4)}(z)
{(1+O(|z|^L))}
\eeal
{for all $L>0$.}
Note that both $f$ and $g$ are holomorphic functions on the domain
	\beal\label{eq:domain}
		\C\smallsetminus \left(\{\varphi i\R_{<0}\}\bigcup_{0\neq n\in\Z} \{n+\varphi i\R_{>0}\}\right)
	\eeal
as $\Li_2^\varphi(z)$ jumps by $2\pi i\log z$ when $z$ crosses the cut on $\re(z/\varphi)=0, \re(z)>0$. Moreover, $e^{\frac{\Li_1(-e^{-2iv})}2}$ changes the sign when $v$ crosses the line $\re(z/\varphi)=0,\ \re(z)>0$. The contour and the branch cuts of $f$ are plotted in Figure \ref{fig:contourS_asymp}.

\begin{figure}[h]
\centering 
\input{contourS1.tex}
\caption{The contours $\mathcal S$, parameterized by $\gamma$, $\gamma'$ (after applying Lemma \ref{asymp_lemma}).}
\label{fig:contourS_asymp}
\end{figure}
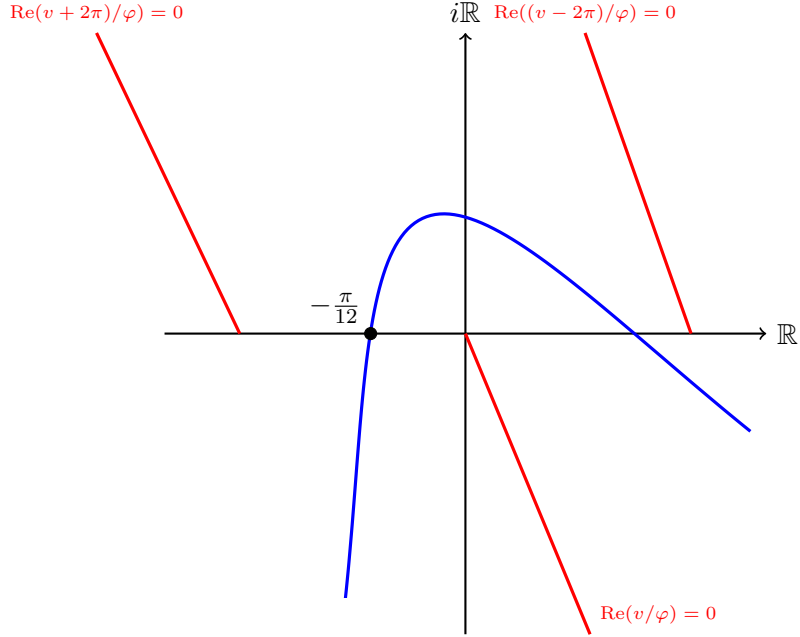

We compute 
\begin{equation*}\begin{aligned}
	f'(v) &\=
	  -\frac i 2 \log(1-e^{-4iv})
	  +v
	  -\frac{\pi}4 - {\sign(\Re(v/\varphi))} \frac{\pi} 2\\
	f''(v) &\= \frac{1+e^{-4iv}}{1-e^{-4iv}}
\end{aligned}\end{equation*}
and thus the critical points $v_0$ of $f$ satisfy
\begin{equation}\label{eq:nahmequation}
(1-e^{-4iv_0})^2 \= -e^{-4iv_0},
\end{equation}
in other words $e^{-4iv_0} = e(\pm\frac{1}6)$.  Using this along with the condition $f'(v) =0$ at saddle points, one may easily classify all saddle points of $f$.

Since we are interested in saddle points between the two branch cuts in the upper half-plane emanating from $-\pi/2$ and $0$ we easily check that there are two saddle points which maximize the real part of $f$, given by $-5\pi/12$ and $-\pi/12$. Away from the branch cuts, the function $f$ is holomorphic in the complex plane, and so in particular we may deform the contour of integration by Cauchy's theorem. We choose the contour passing through the saddle point $-\pi/12$ given by the union of two pieces - again noting that by Cauchy's theorem we are free to choose our path of integration as long as it remains in the region where the integrand is holomorphic. 
The first piece is given by $\gamma(s) = -\pi/12 - s\sqrt{i\varphi}$ for $s \in (-c_0, \infty)$ where $c_0 =\pi/2 - \varepsilon$ for some fixed choice of $\varepsilon>0$. Let $\varphi = -ie^{it}$ and $d(x) \coloneqq -\frac{\pi}{12} + \frac{x}{ \Re(\sqrt{i\varphi})} \sqrt{i\varphi}$, noting that $d(\pi/12)$ is the point at which $\gamma(s)$ intersects the line $i\R \varphi$. The second piece is given piecewise by 
\begin{align*}
	\gamma'(s) = \begin{cases}
		d(\pi/12)  -i s\sqrt{i \varphi}, &\text{ if }  0< t < \pi/2,\\
		d(\pi/12) -i s \sqrt{-i/\varphi}, &\text{ if } \pi/2 \leq t < 3\pi/4,\\
		d(\pi/4) -i s \sqrt{-i/\varphi}, &\text{ if } 3\pi/4 \leq t < \pi.
	\end{cases}
\end{align*}

	We now aim to show that $\Re(f(\gamma(s))/\varphi)$ is maximized at $s=0$, i.e. at the saddle point, so that $\gamma(s)$ is a stationary contour. For this, we first rewrite
\begin{align*}
	\Re\left(\frac{f(\gamma(s))}{\varphi} \right) = s^2 	\Re\left(\frac{f(\gamma(s))}{s^2 \varphi} \right).
\end{align*}
Now making the change of variable $s \sqrt{i \varphi} = \mathfrak{z}$ we obtain
\begin{align*}
	s^2 	\Re\left(\frac{i f(-\pi/12 - \mathfrak{z})}{\mathfrak{z}} \right).
\end{align*}
The function inside the real part is holomorphic and non-constant on the open domain $D_T \coloneqq \{\mathfrak{z} \in \C \colon \Re(\mathfrak{z} > 0, \Im(\mathfrak{z})>0, |z| < T)\}$ by construction, and so the maximum modulus principal for harmonic functions implies that it takes its maximal real part on the boundary of $D_T$. This occurs when $\varphi = \pm i$, when $|\mathfrak{z}| = s =T$ (and we take a limit $T \to \infty$), or when $s =0$.

Firstly, using e.g.\@ SageMath \cite{sage}, we may explicitly compute a (complicated) expression for $\Re(\frac{f(\gamma(s))}{\varphi})$. Using this along with the fact that $\Li_2(z) \sim z$ as $|z| \to 0$ we have that $\Re(\frac{f(\gamma(s))}{\varphi}) \sim - c s$ for some positive constant $c$ as $s \to \infty$. This means that on the boundary piece $|\mathfrak{z}| = s =T$ as $T \to \infty$ that the real part is strictly smaller than at the saddle point.

Now, when $s=0$ for any $\varphi$ we have equality, since we are at the saddle point itself. 

Finally, we consider the cases where $\varphi = \pm i$. First let $\varphi = -i$. A direct calculation (again using SageMath) shows that there is equality of maximal real part to that at the saddle point for certain $s>0$ (but it is never exceeded). Another direct calculation shows that on $\varphi = i$ and $s>0$ the real part is strictly smaller than at the saddle point. Noting that we do not allow these two boundary cases in our application shows that $\gamma(s)$ is a stationary contour for $s\in (0,\infty)$.

For $s \in (-c_0,0)$ the argument is similar, simply noting that a direct calculation for $s = -c_0$ yields a strictly smaller real part than at the saddle point, and so overall we obtain that $\gamma(s)$ is a stationary contour for $s \in (-c_0,\infty)$. The calculations for the second piece of the contour are similar, using the maximum modulus principal on each piece.

We write $Q \coloneqq e^{-4iv_0} = e(1/6) = \frac{1+\sqrt 3 i}2$.
 If we parameterize $\mathcal S$ as $v=-\frac \pi {12} +is\sqrt z$ in a small neighborhood around $v_0$ we obtain that the contribution corresponding to the stationary point $v_0 = -\frac \pi {12}$ is given by
 
\bea[]
&\frac{-1}{2\sqrt{z}}
e^{\pi^2/48{z}+{f}(v_0)/z{-z/{12}}}
\int
\exp
\biggl(-f''(-\tfrac \pi {12}) \frac {s^2}2
+\sum_{l\geq 3}
\frac{f^{(l)}(Q)}{l!} (is)^lz^{l/2-1}
\biggr)
g\Bigl(-\frac \pi {12} +is\sqrt z;z\Bigr)
ds\\
\=&\sqrt{\frac{2\pi i}{-z}}
e^{\pi^2/48{z}+{f}(v_0)/z}\; 
\gamma_{(1/4)}^{-}(z)
\eea
where
\bea
\frac{\pi^2}{48}+f(v_0)
&\=\frac{\pi^2}{48}-\frac{\Li_2(e^{-4iv_0})}8+\frac{v_0^2}2+\frac{\pi v_0}4
\=\frac{D(e(1/6)) i}8 
\= -V \= -0.1268677\ldots i.
\eea
With the definition of $\psi_{-1}(4z;-e^{-2iv})$ from Lemma~\ref{lma:asymp_lemmaqpoch} the power series
$\gamma_{(1/4)}^{-}(z)\in\C[[z]]$ is defined as a formal Gaussian integration by
\beal\label{eq:defgammam}
\gamma_{(1/4)}^{-}(z)
\=
&\frac{1}{2\sqrt{-2\pi i}}
e^{-\pi i/24-z/12}\\
&\qquad\times\int
\exp\biggl(
-\sqrt 3 i\frac{s^2}2 + \frac{is\sqrt z}2
+\frac 1 8 \sum_{l\geq 3} \Li_{2-l}(Q)
\frac{(4s\sqrt z)^l}{zl!}\\
&\qquad\qquad\qquad
-\sum_{t=1}^2 \sum_{k\geq 1, r\geq0}
B_k\biggl(1-\frac t2\biggr)
\Li_{2-k-r}(-(-1^t)\sqrt Q)
\frac{(2s\sqrt z)^r(4z)^{k-1}}{r!k!}
\biggr)ds.
\eeal
The first coefficients are given by
\beal
\gamma_{(1/4)}^{-}(z)
\= \frac 1 {2\sqrt[4]{3(2-\sqrt 3)}}
\left(
1
+\left(\frac 1 3+\frac {77}{216}\sqrt{3}\right)iz
-\left(
\frac{89449}{31104}
+\frac{647}{648}\sqrt 3
\right) z^2+O(|z|^3)
\right).
\eeal
Putting everything together, we obtain
the asymptotic expansion
\bea
v_1^{[0]}(q)
&\= \sqrt {\frac{2\pi i}{-z}} e^{-V/z }\; \gamma^-_{(1/4)}(z)
{(1+O(|z|^L))}
\;
	 +\phi^{[0]}_{(1/4)}(z)
{(1+O(|z|^L))}
\eea
 
for all $L>0$.  
If $\arg{\varphi} >0$, the exponential contribution is the biggest term in (\ref{final_asymp_v10}) and for $\arg{\varphi} <0$, the power series $\phi^{[0]}_{(1/4)}(z)$ has the largest contribution.
\end{proof}

\subsubsection{Sum over odd parts}\label{sec:Sum over odd parts}
\begin{proof}[Proof of (\ref{final_asymp_v11})]
The asymptotics of $v_{1}^{[1]}(q)$ as defined in (\ref{eq:defv10v11}) as $q\to i$ is similar.

{We change the contour in the integral representation in Lemma \ref{int_repv1} to a stationary contour $\mathcal S$.}
After applying the asymptotics from Lemma \ref{asymp_lemma} and Lemma \ref{lma:asymp_lemmaqpoch} and following a similar argument as in Section \ref{sec:sumeveni} we obtain
{that $v_1^{[1]}(q)$ is asymptotically equal to}
\bea
\frac {-e^{\pi i 3/4}}{2z} e^{\pi^2/48z{-z/12}}
\int_{\mathcal S}
&\sign(\re(v/\varphi))
\exp\left(
\frac{
-\Li_2(e^{-4iv})+4v^2+2\pi v
-4\sign(\Re(v/\varphi))\pi v}
{8z}
\right)\\
&\exp\left(
-\frac{iv}2
+ \frac{\Li_1^\varphi(e^{-2iv})}{2}
+\psi_{-1}(4z;e^{-2iv})
\right)
dv
+\phi_{(1/4)}^{[1]}(z),
\eea

where
\bea
	\phi_{(1/4)}^{[1]}(z)
	\=&\sum_{\substack{n<0 \\ n \equiv 1 \bmod 2}}
	\frac{q^{n(n+1)/2}}{(-q^2;q^2)_n}
	\= \sum_{\substack{l>0:\\ l\equiv1\bmod 2}}
			q^{-l(l-1)/2}(-1;q^2)_l
		\=2\sum_{\substack{m\geq 0:\\ m\equiv0\bmod 2}}
			q^{-m(m+1)/2}(-q^2;q^2)_m\\[7pt]
	\=
	 &2 + 8 i z - 96 \, z^{2}
	 - \frac{5708}{3} i \, z^{3}
	 + 52640 \, z^{4} 
	 + \frac{28056121}{15} i \, z^{5}
	 - \frac{405909568}{5} \, z^{6} \\
	 &- \frac{2622584263067}{630} i \, z^{7}
	 + \frac{5171242573856}{21} \, z^{8}
	 +\frac{748741881749741041}{45360} i \, z^{9} 
	 +O(z^{10}).
\eea
We define
\bea
	f(v) &\= -\frac{\Li^\varphi_2(e^{-4iv})}8+\frac {v^2}2+\frac{\pi v} 4 - \sign\left(\re\left( v /\varphi\right)\right) \frac{\pi v}2,\\
g(z;v) &\={
\sign(\re(v/\varphi))
\exp\left(
-\frac{iv}2
+ \frac{\Li_1^\varphi(e^{-2iv})}{2}
+\psi_{-1}(4z;e^{-2iv})
\right)
}
\eea
where $f$ and $g$ are  holomorphic functions on the domain defined in (\ref{eq:domain}) for the same reason as in Section \ref{sec:sumeveni}.
Then
\bea
v_1^{[1]}(q) &\= 
\frac {-e^{\pi i 3/4}}{2z} e^{\pi^2/48z{-z/12}}
\int_{\mathcal S}
e^{f(v)/z} g(z;v)
dv
{(1+O(|z|^L))}
+\phi_{(1/4)}^{[0]}(z)
{(1+O(|z|^L))}
 \eea
{for all $L>0$}
and we compute
\bea
	f'(v) &\=
	-\frac i 2 \log(1-e^{-4iv})
	+v+ \frac{\pi v} 4
	- \sign\left(\re\left( v/ \varphi\right)\right)
	\frac{\pi v}2\\
	f''(v) &\= \frac{1+e^{-4iv}}{1-e^{-4iv}}
\eea
such that the unique stationary point of $f$ is $v_0 = \frac \pi {12}$.  Similar arguments as above using the saddle-point method imply that the contribution corresponding to the stationary point $v_0 = \frac \pi {12}$  is given by 
\bea[]
&e^{\pi^2/48{z}+{f}(v_0)/z{-z/{12}}}
\int
\exp
\biggl(-f''(\tfrac \pi {12}) \frac {s^2}2
+\sum_{l\geq 3}
\frac{f^{(l)}(Q)}{l!} (is)^lz^{l/2-1}
\biggr)
g\Bigl(\frac \pi {12} +is\sqrt z;z\Bigr)
ds\\
&\=e^{V/z }
\sqrt{\frac{2\pi i}{z}}\; \gamma^+_{(1/4)}(z)
\eea
where 
\beal\label{eq:defgammam2}
\gamma_{(1/4)}^{+}(z)
\=
&\frac{1}{2\sqrt{2\pi}}
e^{-\pi i/24-z/12}\\
&\qquad\times\int
\exp\biggl(
\sqrt 3 i\frac{s^2}2 + \frac{is\sqrt z}2
+\frac 1 8 \sum_{l\geq 3} \Li_{2-l}(Q)
\frac{(4s\sqrt z)^l}{zl!}\\
&\qquad\qquad\qquad
-\sum_{t=1}^2 \sum_{k\geq 1, r\geq0}
B_k\biggl(1-\frac t2\biggr)
\Li_{2-k-r}((-1^t)\sqrt Q)
\frac{(2s\sqrt z)^r(4z)^{k-1}}{r!k!}
\biggr)ds.
\eeal
The first coefficients are given by
\bea
\gamma^{+}_{(1/4)}(z)\=
\frac 1 {2\sqrt [4] {3(2-\sqrt 3)}}
\left(
1+
\left(\frac 1 3-\frac{77}{216}\sqrt3 \right)iz
+\left(-\frac{89449}{31104}+\frac{647}{648}\sqrt 3\right)z^2+O(|z|^3)
\right).
\eea
Putting everything together, we obtain
the asymptotic expansion
\bea
v_1^{[1]}(q)
&\= \sqrt {\frac{2\pi i}{z}} e^{V/z }\; \gamma^+_{(1/4)}(z)
(1+O(|z|^L))
\;
+\phi^{[1]}_{(1/4)}(z)
 {(1+O(|z|^L))}
\eea
for all $L>0$
which completes the proof.
\end{proof}

\subsection{Proof of Theorem \ref{Thm: radial asymp}, (2)}\label{sec:pfasympv12}
Throughout let $\zeta = e(\alpha) = e(r/m)$ be a root of unity of order $m$ divisible by $4$.
As in Section \ref{sec:sumeveni} we split up the sum 
defining $v_1(q)$ depending on $n\bmod m$ 
 
\begin{equation}\label{eq:splitv1}
v_1(q) \;=\sum_{n_0 = 0}^{\frac m 2-1} v_1^{[n_0]}(q),
\end{equation}
where
\begin{equation}\label{eq:defv10v11}
v_1^{[n_0]}(q) \; = \sum_{\substack{n\geq 0\\n \equiv n_0 \bmod \tfrac m2}} \frac {q^{n(n+1)/2}}{(-q^2;q^2)_n}.
\end{equation}
for $n_0\in\{0,\ldots, \tfrac m 2-1\}$.
We start with the following lemma.
\begin{lemma}\label{lemma:zetatothenn}Let $\zeta = e(\alpha) = e\left(\frac r m\right)$ be as above.
	\begin{enumerate}
		\item Choose $\overline r \in \Z$ with $\overline r = r \bmod 4$. If $n = n_0 \bmod \frac m2$ then
			\bea
				\zeta^{n(n+1)/2}
				\= \zeta^{n_0(n_0+1)/2}
				e\left(\frac{n-n_0}4
				+(-1)^{n_0}\frac {\overline r} 2  \frac{n-n_0}m\right).
			\eea
		\item Let $\oor = \pm 1$ such that $\oor =r\bmod 4$. Then for all $n = n_0 \bmod \frac m2$
			\bea
				\zeta^{n(n+1)/2} (-1)^{2 (n-n_0)/m}
				= \zeta^{n_0(n_0+1)/2}
				e\Big((-1)^{m/4+n_0+1}\ \oor\ \frac {n-n_0}{2m}\Big).
			\eea
	\end{enumerate}
\end{lemma}

\begin{proof} We begin by proving part (1).  
\begin{enumerate}
\item We write with $n=km/2+n_0$, $k= 2\frac{n-n_0}{m}$,
	\bea
		\zeta^{n(n+1)/2}
		\= &e\left({\alpha n(n+1)}/2\right)\\
		\= &e\left(\alpha \left(\tfrac {km} 2+n_0\right) \left(\tfrac{km}2+n_0+1\right)/2\right)\\
		\= &e\left(\alpha \left(\tfrac{km}2+n_0\right)^2/2 + \alpha\left(\tfrac{km}2+n_0\right)/2\right)\\
		\= &e\left(\alpha \left(
		\tfrac{k^2 m^2}8 + \tfrac{km n_0}2 +  \tfrac{n_0^2}2 + \tfrac{k m}4 +  \tfrac{n_0}2\right)\right).
	\eea
	Note that $\frac{\alpha m^2}8 \in \frac 1 2 \Z$ and thus $e(\alpha\frac{k^2 m^2}8)=e(\alpha \frac{k m^2}8)$ as $k^2 = k \bmod 2$. Hence, we obtain
	\bea
		\zeta^{n(n+1)/2} 
		\= &e\left(\alpha k\left(\frac{m^2}8 + \frac{mn_0}2 +  \frac m 4\right)\right)\; e\left(\alpha \frac{n_0(n_0+1)}2\right)\\
		\= &e\left(\alpha k m\left(\frac m 8 + \frac{n_0} 2 +  \frac 1 4\right)\right)\; \zeta^{n_0(n_0+1)/2}.
	\eea
Moreover, we compute
\bea
	e\left(\alpha k m\left(\frac m 8 + \frac{n_0} 2 +  \frac 1 4\right)\right)
	\=&e\left(\alpha \frac {k m} 4\left(\frac m 2 + 2n_0 +  1\right)\right)\\
\eea
and note that the denominator of $\alpha \frac {k m} 4$ is either $1$ or $4$. If $n_0$ is even, $2n_0$ is divisible by $4$ and $e(\alpha \frac {k m} 4 2n_0)) =1$. Otherwise, $2n_0+2$ is divisible by $4$ and thus $e(\alpha \frac {k m} 4 (2n_0+1)) =e(-\alpha \frac {k m} 4)$. In other words,
\bea
	e\left(\alpha \frac {k m} 4\left(\frac m 2 + 2n_0 +  1\right)\right)
	\=&\begin{cases}
		e\left(\alpha \frac {k m} 4\left( \frac m 2 + 1\right)\right), & \text{if $n_0$ is even,}\\
		e\left(\alpha \frac {k m} 4\left( \frac m 2 - 1\right)\right), & \text{if $n_0$ is odd,}\\
	\end{cases}\\
	\=&e\left(\alpha \frac {k m} 4\left( \frac m 2 + (-1)^{n_0}1\right)\right)\\
	\=&e\left(\alpha \frac{km^2}8\right)\ e\left((-1)^{n_0}\alpha \frac{mk}4\right).
\eea
Note that $\alpha \frac{km^2}8\in \frac 1 2 \Z$, hence 
\bea
	e\left(\alpha \frac{km^2}8\right)
	\= e\left(\frac{km}8\right) \= e\left(\frac {n-n_0}4\right).
\eea
Moreover, we note that $\alpha \frac{mk}4$ has denominator $4$ such that 
\bea e\Big((-1)^{n_0}\alpha \frac{mk}4\Big) = e\Big((-1)^{n_0}\overline r \frac k 4\Big)=e\Big((-1)^{n_0}\frac {\overline r} 2  \frac{n-n_0}m\Big).\eea
This proves the first part of the lemma. \ \\

\item To prove part (2), we choose $\overline r$ with $\overline r = \oor +4l$, $l\in\Z$ such that
\bea
	(-1)^{n_0} 2l + \frac m4-1 \=
	\begin{cases}
		0,	& \text{if $\frac m 4$ is odd,}\\
		-(-1)^{n_0}\oor, & \text{if $\frac m 4$ is even.}\\
	\end{cases}
\eea
We continue with the notation from above with $(-1)^{2 (n-n_0)/ m} = e(-\frac k 2)$:
\bea
	e\left(\tfrac{km}8\right)
	e\left((-1)^{n_0}\tfrac{\overline r k}4\right)e\left(-\tfrac k 2\right)
	\=&e\left(
		\tfrac{km}8
		+(-1)^{n_0}\tfrac{\overline r k}4
		-\tfrac k 2\right)\\
	\=&e\left(
		k\; \left(\tfrac m 4+(-1)^{n_0}\tfrac {\overline r }2-1\right)/2\right)\\
	\=&e\left(k\; \left(\tfrac m 4+(-1)^{n_0}\tfrac{\oor+4l}2-1\right)/2\right)\\
	\=&e\left(k\; 
	\left(\tfrac m 4+(-1)^{n_0}\tfrac \oor 2+(-1)^{n_0}2l-1\right)/2\right).
\eea
Using the choice of $\overline r$, the last expression becomes
\bea[]
	\begin{cases}
		e\left(k\; (-1)^{n_0}\frac \oor 4\right),
			&\text{if $\frac m 4$ is odd,}\\[6pt]
		e\left(k\; \left((-1)^{n_0}\frac \oor 2 - (-1)^{n_0}\oor\right)/2\right),
			&\text{if $\frac m 4$ is even,}
	\end{cases}
\eea
which in both cases is equal to
\bea[]
	e\left(k\ (-1)^{m/4-1+n_0}\;\frac{\oor}4\right)
	\= & e\left((-1)^{m/4-1+n_0}\;\oor \frac {n-n_0}{2m}\right).
\eea \qedhere
\end{enumerate}
\end{proof}

{
From now on we write
\bea
	 \delta \;:=\;(-1)^{m/4+n_0+1}\ \oor
\eea}
with $\oor=\pm 1$ as in the previous lemma.

\subsubsection{Integral representation}
\begin{lemma}\label{int_repv1n0}
	Let $L_\infty$ be the contour depicted in Figure \ref{figure_contour}. Then with $q=\zeta e^{-z}$,
\bea
v_1^{[n_0]}(q)
	&\=
	\frac
		{-\zeta^{n_0(n_0+1)/2}
			e({ \delta} \tfrac {n_0}{2m})}
		{m(-q^2;q^2)_\infty}
	\int_{L_\infty}
		e^{\pi { \delta}  {t}/{m}}
		e^{zt^2/2} e^{-izt/2}
		\frac {(-\zeta^{2n_0}e^{-2zit}q^2;q^2)_\infty}{\sin(\pi 2(s-n_0)/m)}
	dt.
\eea
\end{lemma}

\begin{proof}
First, note that we have   
\bea
	v_1^{[n_0]}(q)
		&\=\frac {\zeta^{n_0(n_0+1)/2}} {(-q^2;q^2)}
	\sum_{\substack{n\geq 0\\ n=n_0\bmod \frac m2}}
		e({ \delta} \tfrac {n-n_0}{2m})
		{e^{-hn(n+1)/2}}\
		{(-\zeta^{2n_0} e^{-h2n}q^2;q^2)_\infty}
\eea
 using Lemma~\ref{lemma:zetatothenn}, (2).  
The function $\frac 1 {\sin\left(
		2\pi (s-n_0)/m
	\right)}$
has poles at $s \in \Z$ with $s = n_0 \bmod \frac m2$ and residues $(-1)^{2(n-n_0)/m}\ \frac m {2\pi}$.
Hence, we write with
Cauchy's residue theorem
\begin{align*}
	v_1^{[n_0]}&(q)  \\
	&\=\frac {-1}{2\pi i} \frac 1 {(-q^2;q^2)_\infty}
	\lim_{R\to \infty}\int_{L_R+C_R}\!\!\!
		{\zeta^{s(s+1)/2}}
		e^{-hs(s+1)/2}
		(-\zeta^{2n_0}e^{-2hs}q^2;q^2)_\infty
		\frac {2\pi (-1)^{2(s-n_0)/m}}{m\sin(\pi 2(s-n_0)/m)}
	ds\\[10pt]
	&\=
	\frac
		{-\zeta^{n_0(n_0+1)/2} 
			e({ \delta} \tfrac {n_0}{2m})}
		{im(-q^2;q^2)_\infty}
	\lim_{R\to \infty}\int_{L_R+C_R}\!\!\!
		e^{\pi i { \delta} {s}/{m}}
		e^{-hs(s+1)/2}
		\frac {(-\zeta^{2n_0}e^{-2hs}q^2;q^2)_\infty}{\sin(\pi 2(s-n_0)/m)}
	ds.
\end{align*}
The convergence follows analogously to Lemma \ref{int_repv1}.
\end{proof}

\subsubsection{Proof of Theorem \ref{Thm: radial asymp}, (2)}

We will prove the following result, Proposition \ref{Prop} below. The proof of Theorem \ref{Thm: radial asymp}, (2) 
 and the asymptotics of $v_1(q)$ follow  then by summing  $v_1^{[n_0]}(q)$ over $n_0= 0,\ldots, \frac m 2 {- 1}$ as in \eqref{eq:splitv1}  
and setting  
\bea
	\gamma^{\pm}_{(\alpha)} \= \sum_{\substack{n_0 =  0,\ldots, \frac m 2\\ \delta = \pm1}}
	\gamma_{(\alpha)}^{[n_0]}.
\eea

\begin{prp}\label{Prop} Let $\zeta = e(\alpha)$ be a root of unity of order $m$.
	For $n_0\in\{0,\ldots, \tfrac m 2-1\}$ we have
	\bea
		v_1^{[n_0]}(q)
		\;&=\;
		e^{\delta \frac {{ 16} V}{zm^2}}\;
		\left(\frac{ z}{2\pi i}\right)^{-1/2}\!
		\gamma_{(\alpha)}^{[n_0]}{(z)}
 {(1+O(|z|^L))}
		+\phi_{(\alpha)}^{[n_0]}(z)
 {(1+O(|z|^L))}
	\eea
 {for all $L>0$}
	as $q = \zeta e^{-z} \to \zeta$, where $\gamma_{(\alpha)}^{[n_0]}(z) {\in\C[[z]]}$   is defined in  (\ref{eq:defgamman0}) and 
	\bea
		\phi_{(\alpha)}^{[n_0]}(z) \= \sum_{\substack{n<0\\ n\equiv n_0 \bmod \frac m 2}} \frac{q^{n(n+1)/2}}{(-q^2;q^2)_n}.
	\eea
\end{prp}
\begin{proof}
{Throughout we write $z=\varphi h$ where $\varphi \in \C$ with $|\varphi|=1$ and $0\neq|\arg(\varphi)| < \frac \pi 2$.}
We substitute $s = iv/z$ in the integral representation from Lemma \ref{int_repv1n0} to obtain

\bea
v_1^{[n_0]}(q)\= 
	\frac 
		{-\zeta^{n_0(n_0+1)/2}
			e({ \delta} \tfrac {n_0}{2m})
		}{mh(-q^2;q^2)_\infty}
	\int_{-ihL_\infty}
		e^{\pi { \delta}  {v}/{mz}}
		e^{v^2/2z} e^{-iv/2}
		\frac {(-\zeta^{2n_0}e^{-2iv}q^2;q^2)_\infty}{\sin(\pi 2(iv/z-n_0)/m)}
	dv.
\eea

Changing the contour of integration to a stationary contour $\mathcal S$, we include the poles at $-2iz\Z_{<0}$ whose residues give a power series
\bea
	\phi^{[n_0]}_{(\alpha)}(z) \= \sum_{\substack{n<0 \\ n \equiv n_0 \bmod \frac m 2}}
	\frac{q^{n(n+1)/2}}{(-q^2;q^2)_n}
.
\eea
Applying the asymptotics from Lemma \ref{asymp_lemma} and Lemma \ref{lma:asymp_lemmaqpoch} to the integrands, we obtain that $v_1^{[n_0]}(q)$ is equal to

\bea
	\frac 
		{-\zeta^{n_0(n_0+1)/2}
			e({ \delta} \tfrac {n_0}{2m})
		}{mhQ(\zeta^2)}
	\int_{\mathcal S} e^{f(v)/z}
 {    g(z;v)}
	dv
	+\phi_{(\alpha)}^{[n_0]}(z)
\eea
 to all orders 
where $Q(\zeta^2)$ is defined in (\ref{eq:Qdef}) and
\bea
	f(v) &= -\frac{2\Li_2^{\varphi}(e^{-miv})}{m^2}
	+\frac {v^2}2 -\frac{\pi \delta v{ -}\sign\left(\re\left( v/\varphi\right)\right) 2\pi v}{m},\\
	g(z;v) &= \sign(\re(v/\varphi))\exp\!\Biggl(\!
	-\frac{iv}2- \sign(\re(v/\varphi))2\pi i m n_0
	- \frac  {\Li_1^{\varphi}(e^{-miv})}2\\
	&\hspace{150pt}
    +\sum_{t=1}^{m/2}\frac{2t}m
	\Li_1^{\varphi}(-\zeta^{2t+2n_0}e^{-2iv})
 {
+\psi_{\zeta^2}(mz;-\zeta^{2n_0}e^{-2iv})}
\Biggr).
\eea
We recall that $\Li^{\varphi}_2(e^{-miv})$ jumps by $2\pi mv$ when $v$ crosses the branch cut at $\re(v/\varphi) =0$. Hence, the function $f(v)$ is holomorphic on the domain defined in (\ref{eq:domain}). A similar argument shows that $g(v)$ is holomorphic on the same domain.

The stationary points $v_0$ of $f$ are given by
\bea
		f'(v_0) &\= -\frac{2i\log(1-e^{-miv_0})}{m}
	+v_0 -\frac{\pi \delta -\sign\left(\re\left( v_0/\varphi\right)\right)2\pi}{m} \= 0.
\eea
This implies in particular $(1-e^{-iv_0m})^2 = -e^{-iv_0m}$, i.e. $e^{-miv_0} = e(\pm 1/6)$
and it can be checked that $v_0 = \delta \frac{\pi}{3m}$ is the unique stationary point.\\

Applying the saddle-point method and using $\frac{\pi^2}{3m^2}+f(v_0) \= \delta \frac{{ 16} V}{m^2}$ implies
 {that the contribution coming from the saddle point is given by
\bea[]
&\frac
{-\zeta^{n_0(n_0+1)/2}
e({ \delta} \tfrac {n_0}{2m})
}{mhQ(\zeta^2)\sqrt{z}}
e^{\delta 16V/m^2z-z/12}\\
&\qquad\quad\times\int
\exp
\biggl(-f''(v_0) \frac {s^2}2
+\sum_{l\geq 3}
\frac{f^{(l)}(e^{-miv_0})}{l!} (is)^lz^{l/2-1}
\biggr)
g\Bigl(v_0 +is\sqrt z;z\Bigr)
ds\\
\=&e^{\frac{\delta V}{zm^2}}
	\sqrt{\frac{2\pi}{\delta z}}
	\gamma_{(\alpha)}^{[n_0]}(z)
\eea  
where the integral goes through a small neighborhood of $v_0$.
In other words, we have the asymptotic expansion}
\bea
	v_1^{[n_0]}(q) \= e^{\frac{\delta V}{zm^2}}
	\left(\frac{\delta z}{2\pi}\right)^{-1/2}
	\gamma_{(\alpha)}^{[n_0]}{(z)}
 {(1+O(|z|^L))}
	+\phi_{(\alpha)}^{[n_0]}(z)
 {(1+O(|z|^L))}
\eea
 {for all $L>0$}
where 
\beal\label{eq:defgamman0}
	\gamma_{(\alpha)}^{[n_0]}
	\=
	-\frac 
		{\zeta^{n_0(n_0+1)/2}
			e({ \delta} \tfrac {n_0}{2m})2g(v_0)
		}{mQ(\zeta^2)\sqrt{f''(v_0)}}
\eeal
with
\bea
	f''(v_0) \= \frac{1+e^{-miv_0}}{1-e^{-miv_0}},
\eea
which completes the proof.
\end{proof}

\section{Proof of Theorem \ref{Conj3}}\label{sec_3}

\subsection{Wright's Circle Method}\label{Sec: WCM}

In this section we prove Theorem \ref{Conj3} using Theorem \ref{Thm: radial asymp}. We follow the idea of Wright's Circle Method (see Section \ref{Sec: WCM prelim}), and also make use of the saddle-point method (see Section \ref{sec_sp}).   We label two major arcs near $\pm i$ as $C_{1}$ and $C_{2}$, where $C_{1}$ is the arc near $i$, and $C_{2}$ is the arc near $-i$. Using Cauchy's theorem, we recover our coefficients as
\begin{align*}
	V_1(n) \= \frac{1}{2 \pi i} \int_C \frac{v_1(q)}{q^{n+1}} dq,
\end{align*}
where $C$ is a circle of radius less than 1 traversed exactly once in the counter-clockwise direction. We split the integral above into three pieces,
\begin{align*}
	\int_C \= \int_{C_{1}} + \int _{C_{2}} + \int _{C - \cup C_{j} }.
\end{align*}
We denote the asymptotic contribution of the first two integrals $M(n)$ (the main term), and the contribution of the final integral $E(n)$ (the error term).

\subsection{Major arc estimates}\label{Sec: major arc}
In this section we calculate the asymptotic contribution arising from the major arcs. We elucidate explicitly the case of one of the contributing terms on the major arc - the calculations for all other contributions are very similar.
 
Consider the term $M_{1}(n):= \frac{1}{2\pi i} \int_{C_{1}} \frac{v_1(q)}{q^{n+1}} dq$. Choose the radius of the circle $C$ to be $e^{-\lambda}$ with $\lambda \coloneqq \sqrt{\frac{|V|}{n}}$. Then the arc $C_{1}$ is described by $ie^{-\lambda + i \theta}$ with $\theta \in (-\delta,\delta)$ for some parameter $\delta >0$. Therefore we make the change of variable $q=ie^{-z}$ and parameterize where $z$ runs from $\lambda+i\delta$ to $\lambda - i\delta$, to obtain 
\begin{align}\label{eqn_M1n1}
	M_{1}(n) \= - \frac{(-i)^n}{2\pi i} \int_{\lambda+i\delta}^{\lambda - i\delta} \frac{v_1\left( ie^{-z} \right)}{e^{-zn}} dz \= \frac{(-i)^n}{2\pi i} \int^{\lambda+i\delta}_{\lambda - i\delta} \frac{v_1\left( ie^{-z} \right)}{e^{-zn}} dz.
\end{align}

From Theorem \ref{Thm: radial asymp} we have that
\begin{align}\label{eqn_v1asyshort} v_1(i e^{-z}) \= e^{\frac{V}{z}} \left(\frac{z}{2\pi i }\right)^{-\frac{1}{2}}\gamma^+ + e^{-\frac{V}{z}} \left(\frac{-z}{2\pi i }\right)^{-\frac{1}{2}}\gamma^- + e^{\frac{V}{z}} O(|z|^{\frac{1}{2}}) + e^{-\frac{V}{z}}O(|z|^{\frac{1}{2}}).
\end{align}  
The first main term of the asymptotics in \eqref{eqn_v1asyshort} yields the following contribution to \eqref{eqn_M1n1}: 
\begin{align}\label{eqn: integral before Mathematika subst}
	 \frac{(-i)^n \gamma^+}{\sqrt{2\pi i}} \int^{\lambda+i\delta}_{\lambda - i\delta} e^{\frac{V}{z} + nz} z^{-\frac{1}{2}} dz \= 
	 \frac{(-i)^n \gamma^+}{n^{1/4} \sqrt{2\pi i}}  \int_{\sqrt{|V|}(1-i)}^{\sqrt{|V|}(1+i)}  e^{\sqrt{n}\left(\frac{V}{z}+z\right)} z^{-\frac{1}{2}} dz,
\end{align} where we let $\delta=\lambda$ and made the change of variable $z\mapsto  \tfrac{z}{\sqrt{n}}.$   

The integral in  \eqref{eqn: integral before Mathematika subst} is of a shape to which we may apply the saddle-point method. We move the contour in \eqref{eqn: integral before Mathematika subst} through the saddle point $\sqrt{V}$, a zero of the derivative of the function $g(z) \coloneqq  \frac{V}{z}+z$ in the exponential in the integrand.  We label this transformed contour $\Gamma$ so that \eqref{eqn: integral before Mathematika subst} becomes 
\begin{align}\label{eqn: integral before Mathematika subst2} 
	 \frac{(-i)^n \gamma^+}{n^{1/4} \sqrt{2\pi i}}  \int_{\Gamma}  e^{\sqrt{n}\left(\frac{V}{z}+z\right)} z^{-\frac{1}{2}} dz,
\end{align} 
We next make the change of variable $z= \sqrt{V} +  i w n^{-\frac{1}{4}}$, and expand relevant functions around the saddle point:
\begin{align*}
\sqrt{n} g(z) 
&\=  \sqrt{n} \sum_{r=0}^\infty \frac{g^{(r)}(\sqrt{V})}{r!}(i w)^r n^{-\frac{r}{4}}  \\
&\= 2\sqrt{nV} -V^{-\frac12} w^2   + \sum_{r=3}^\infty (-i)^r V^{\frac12(1-r)} w^r n^{\frac{2-r}{4}}
\end{align*}
where we have used that $g(\sqrt{V})=2\sqrt{V},$ $g'(\sqrt{V})=0$, and $g^{(n)}(\sqrt{V}) = (-1)^n n! (\sqrt{V})^{1-n}$ for $n\geq 2$.    Thus, we have that 
\begin{align}\label{eqn_expg} 
e^{\sqrt{n}g(\sqrt{V})} \= e^{2\sqrt{nV}}e^{-V^{-\frac12}w^2} \left(
1+ \sum_{r=1}^\infty n^{-\frac{r}{4}} \widetilde{p}_r(w)
\right), 
\end{align} where each $\widetilde{p}_r(w) \in \mathbb C[w]$.  Similarly, we have that 
\begin{align}\label{eqn_z12exp} z^{-\frac12} \= V^{-\frac14}\left(1 + \sum_{r=1}^\infty \frac{\left(-\frac{iw}{2}\right)^r (2r-1)!!}{r!} V^{-\frac{r}{2}} n^{-\frac{r}{4}}\right).
\end{align}
We use \eqref{eqn_expg} and \eqref{eqn_z12exp} 
in \eqref{eqn: integral before Mathematika subst2}   (with $z= \sqrt{V} +  i w n^{-\frac{1}{4}}$) to obtain
\begin{align}\label{eqn: integral before Mathematika subst3} 
	 \frac{i(-i)^n \gamma^+}{n^{\frac{1}{2}} \sqrt{2\pi i}}  e^{2\sqrt{nV}} V^{-\frac{1}{4}}\int_{\Gamma'}  
	 e^{-V^{-\frac{1}{2}} w^2}\left(1+ \sum_{r=1}^\infty n^{-\frac{r}{4}} p_r(w)\right)  dw,
\end{align} where each $p_r(w) \in \mathbb C[w]$ arises from multiplying the corresponding polynomials in \eqref{eqn_expg} and \eqref{eqn_z12exp}.  
Here, because the contour $\Gamma$ was chosen to run through the saddle point $\sqrt{V}$, the contour $\Gamma'$ runs through the origin.  Moreover, we choose $\Gamma'$ so that it also has a horizontal tangent at $0$.  Then, in the limit as $n\to \infty$, we have that $\Gamma' \to \mathbb R$.  Thus,
the expression in \eqref{eqn: integral before Mathematika subst3} is asymptotic to
\begin{align}
 &\frac{i(-i)^n  \gamma^+}{n^{\frac{1}{2}} \sqrt{2\pi i}}   e^{2\sqrt{nV}} V^{-\frac{1}{4}}\int_{-\infty}^\infty   
	 e^{-V^{-\frac{1}{2}} w^2}\left(1+ \sum_{r=1}^\infty n^{-\frac{r}{4}} p_r(w)\right)   dw \\
	 &\=  \frac{i(-i)^n  \gamma^+}{\sqrt{2 in}}  e^{2\sqrt{nV}} \left(1+O\left(n^{-\frac{1}{2}} \right)\right),
\end{align}
where in the final step we use that the polynomials arising from \eqref{eqn_expg} and \eqref{eqn_z12exp} are odd (resp.\@ even) for $r$ odd (resp.\@ even), and that when multiplied they begin with a term of order $n^{-\frac{1}{2}}$.

The calculations for the contributions of the other terms in  \eqref{eqn_v1asyshort} along with the contributions arising from the major arc around $-i$ are very similar and so we omit them for brevity. Collecting all of the contributions yields
\begin{align*}
M(n) &\=	\left(\frac{(-i)^{n} \sqrt{i} \gamma^+ }{\sqrt{2 n}} e^{2\sqrt{nV}} +  \frac{i^{n-1} \sqrt{i} \gamma^+}{ \sqrt{2 n}} e^{2\sqrt{-nV}}  + \frac{(-i)^n \sqrt{i} \gamma^-}{\sqrt{ 2 n}} e^{2 \sqrt{-n V}}+  \frac{i^{n+1} \sqrt{i} \gamma^-}{\sqrt{2 n}}e^{2 \sqrt{n V}} \right) \\ &\hspace{.15in} \times \left(1+O\left(n^{-\frac{1}{2}}\right)\right).
\end{align*}

Simplifying this yields that
\begin{align}\notag
	M(n)
	&\=
	(-1)^{\lfloor \frac n 2 \rfloor}\ 
{{\re(\sqrt{2 i} (\gamma^+ - (-1)^ni\gamma^-))}}
	\frac {e^{\sqrt{2|V|n}}} {\sqrt{ n}}
	\left(\cos(\sqrt{2|V| n}) +{{(-1)^{n+1}}} \sin(\sqrt{2|V| n})\right) \\ &\hspace{.15in} \times \left(1+O\left(n^{-\frac{1}{2}} \right)\right). \label{eqn: main asymp estimate}
\end{align}
Using that $\sqrt i = \frac {1+i}{\sqrt{2}}$ we see that $M(n)$ gives the first term of Theorem \ref{Conj3}. It remains to estimate the contribution from the minor arcs.

\subsection{Minor arc estimates}
 In this section we bound the asymptotic contribution of the minor arcs, which turn out to almost always be exponentially smaller than those from the major arcs (a fact we prove in Section \ref{Sec: proof of Andrew's conjectures}). 
 
We begin by noting that the asymptotic formula for $v_1$ toward all roots of unity provided by Theorem \ref{Conj3} is valid for all $z$ in any cone contained in the right half-plane. In particular, this means that we have the asymptotic behavior of $v_1$ in a punctured neighborhood inside the unit disk of any root of unity. Since the roots of unity are dense on the unit circle, we thus have asymptotic estimates covering the entire unit circle. This is in essence the estimate one requires for the Circle Method of Hardy and Ramanujan, extended by Rademacher, when taking the path of integration on Farey arcs; see \cite{AndrewsThy, HR, Rad}.

 Recalling Lemma \ref{Lem: bounds at roots of unity not divis by 4}, it suffices to consider primitive roots of unity whose order is divisible by $4$, but which are not $\pm i$. Consider the error term, given by
\begin{align*}
	E(n) \coloneqq \frac{1}{2\pi i}\int_{C - \cup C_{j} } \frac{v_1(q)}{q^{n+1}} dq.
\end{align*}
We have
\begin{align*}
	|E(n)| &\= \frac{1}{2\pi} \left\lvert \int_{C - \cup C_{j} } \frac{v_1(q)}{q^{n+1}} dq \right\rvert.
\end{align*}

By Theorem \ref{Thm: radial asymp} we see that the largest contribution to the error arc is given by the $8$-th order roots of unity. So we may bound the entire error term $E(n)$ by the contribution from the $8$-th order roots of unity multiplied by the length of the integral, which is less than $2\pi$.

Let $\zeta$ be an ${ m}$th root of unity, where ${ 4| m}$ and $m>1$.  Following the saddle-point method in the same way as for the major arcs, and using Theorem \ref{Thm: radial asymp}, we obtain a finite sum of integrals of the form
\begin{align*}\frac{\zeta^{-n} K_\zeta}{\sqrt{2\pi i}} \int_\Gamma e^{ \frac{\pm { 16} V}{m^2 z} + n z} z^{-\frac12} dz  
	&\= \frac{\zeta^{-n}K_\zeta}{\sqrt{2\pi i }} n^{-\frac14} \int_{\Gamma'} e^{\sqrt{n}\left( \frac{\pm { 16} V}{m^2 z} + z\right)}z^{-\frac 12} dz.
\end{align*}  
Here, $K_\zeta \in \C$ is some constant depending on the root of unity $\zeta$ (and also on which term from Theorem \ref{Thm: radial asymp} we are adding). Following the saddle-point method, this yields the contribution
\begin{align}\label{eqn: saddle point other 4k roots}
	K'_\zeta \frac{e^{2\sqrt{\pm \frac{{ 16}n V}{m^2}}}} {n^{\frac12}} \left(1+O\left(n^{-\frac{1}{2}} \right)\right),
\end{align}
for some constant $K'_\zeta \in \C$.  

By \eqref{eqn: saddle point other 4k roots} we get (setting $K'_{\zeta,{ 8}}$ as the constant arising from the $m={ 8}$ term there, i.e. $8$-th order roots of unity)
\begin{align}\label{eqn: error bound}
	\lvert E(n) \rvert &\ll  \left\lvert  K'_{\zeta,{ 8}} e^{\sqrt{\pm nV}} n^{-\frac{1}{2}} \right\rvert \= O\left( n^{-\frac{1}{2}}  e^{\sqrt{\frac{ n|V|}{2}}} \right).
\end{align}
Now combining \eqref{eqn: main asymp estimate} along with \eqref{eqn: error bound} finishes the proof of Theorem \ref{Conj3}.

\section{Proof of Theorem \ref{thm_main}}\label{Sec: proof of Andrew's conjectures} 

\subsection{Proof of Andrews' Conjecture 3}  
 Andrews' Conjecture 3 states that $\lvert V_1(n)\rvert \to \infty$ as $n \to \infty$.  As remarked in Section \ref{sec_intro}, after computational and theoretical investigations, we believe that this conjecture can be refined to say ``We have that $|V_1(n)|\rightarrow\infty$ as $n\rightarrow \infty$ away from a set of density $0$.'' This refined conjecture will follow from the arguments needed to prove Andrews' Conjecture 4 below. 

\subsection{Proof of Andrews' Conjecture 4}
 Recall that Andrews' Conjecture 4 states that four consecutive values of $V_1$ come with two positive and two negative signs almost always.

From Theorem \ref{Conj3}, this reduces to studying the function
\begin{align*}
	 (-1)^{\lfloor \frac{n}{2} \rfloor} \left(\cos\left( \sqrt{2|V|n} \right) + (-1)^{n+1} \sin\left( \sqrt{2|V|n} \right)\right).
\end{align*}
Note that we have the following table of signs for $(-1)^{\lfloor \frac{n}{2}\rfloor}$:
\begin{center}
	\begin{tabular}{c|c} 
		$n \pmod{4}$ & $(-1)^{\lfloor \frac{n}{2} \rfloor}$   \\  
		\hline  $0$ & $+$ \\  \hline 
		$1$ & $+$ \\  \hline 
		$2$ & $-$ \\  \hline 
		$3$ &  $-$
	\end{tabular}
\end{center}
Thus, it is enough to study the function (of $n$)
\begin{align*}
	\cos\left( \sqrt{2|V|n} \right) + (-1)^{n+1} \sin\left( \sqrt{2|V|n} \right)
\end{align*}
at $n, n+1, n+2$, and $n+3$.

Heuristically, when $n$ gets large the values $\cos(\sqrt{2|V|(n+j)})$ (resp. $\sin(\sqrt{2 |V|(n+j)})$) for $j\in\{0,1,2,3\}$ are close to each other.  To see this, for $a \in \R$ consider
\begin{align*}
	\lim_{x \to \infty} \cos(a\sqrt{x+1}) - \cos(a\sqrt{x}) \= & \lim_{x\to \infty} -2 \sin\left(\frac{a (\sqrt{x+1} -\sqrt{x})}{2}\right) \sin\left(\frac{a(\sqrt{x+1} + \sqrt{x})}{2}\right)\\
	\= &   \lim_{x\to \infty} -2 \sin\left(\frac{a }{2(\sqrt{x+1} +\sqrt{x})}\right) \sin\left(\frac{a(\sqrt{x+1} + \sqrt{x})}{2}\right) \\
	\=& 0,
\end{align*}
where the final step arises from the expansion at infinity, which is
\begin{align}\label{eqn: trig bound}
\left\lvert \cos(a\sqrt{x+1}) - \cos(a\sqrt{x}) \right\rvert \leq 2 \left(\frac{a}{4\sqrt{x}} + O\left(x^{-\frac{3}{2}}\right)\right).
\end{align}
Note that a similar calculation holds for the $\sin$ term.

By Theorem \ref{Conj3} we have that 
\begin{align*}
V_1(n) &\= M(n) + E(n) \\
\=&(-1)^{\lfloor \frac n 2 \rfloor}\ 
\frac {e^{\sqrt{2|V|n}}} {\sqrt{n}}
{(\gamma^+ + (-1)^n\gamma^-)}
\left(\cos(\sqrt{2|V| n}) -{{(-1)^{n}}} \sin(\sqrt{2|V| n})\right) \left(1+O\left(n^{-\frac{1}{2}} \right)\right) \notag\\
& + O\left( n^{-\frac{1}{2}} e^{\sqrt{\frac{|V|n}{2}}} \right).
\end{align*}

 We want to prove that almost all of the time the main term $M(n)$ is asymptotically larger than the error $E(n)$. The only time this could not happen is when the factor $\cos(\sqrt{2|V| n}) +(-1)^{n+1} \sin(\sqrt{2|V| n})$ is exponentially small, which in turn can only happen near to roots of
 \begin{align*}
 	\mathcal{F}_{\pm} (x) \coloneqq  \cos(x) \pm \sin(x).
 \end{align*}
Such roots occur at $ \pi \left( \ell \pm \frac{1}{4} \right)$ for any $\ell \in \Z$. In the interval $[0,2\pi]$ we thus have four different roots, which we label by $\vartheta_j$ with $1\leq j\leq4$.  Then taking the Taylor series of $\mathcal{F}_\pm$ about $\vartheta_j$ gives
	\begin{align*}
		(x-\vartheta_j)\mathcal{F}_\pm'(\vartheta_i) + O\left( (x-\vartheta_j)^2 \right),
	\end{align*}
where $|\mathcal{F}_\pm'(\vartheta_i)| = \sqrt{2}$.

Then the part where $M(n) \gg E(n)$ occurs when we have $|\mathcal{F}_{\pm} (x)| > e^{-\kappa\sqrt{n}}$ with $\kappa \coloneqq \sqrt{\frac{|V|}{2}}$. Equivalently, we want the argument of $\mathcal{F}$ to stay $ e^{-\kappa\sqrt{n} + \varepsilon\sqrt{n}}$ with $\varepsilon>0$ away from each $\vartheta_j$, since $| \cos(x_n) \pm \sin(x_n) | > e^{-\kappa \sqrt{n}}$ if
\begin{align*}
	\lvert x-\vartheta_j \rvert + O\left( \left( x-\vartheta_j \right)^2 \right) \= e^{-\kappa\sqrt{n} + \varepsilon\sqrt{n}}+ O\left( \frac{e^{-2\kappa\sqrt{n} + 2\varepsilon\sqrt{n}}}{2} \right) > e^{-\kappa \sqrt{n}}
\end{align*}
for all $j$. 

In what follows, we use an argument based on the equidistribution of sequences modulo $1$ to show that almost always the points $x_n \coloneqq \sqrt{2|V|n}$ are more than $ e^{-\kappa\sqrt{n}+\varepsilon\sqrt{n}}$ away from each $\vartheta_j$. To begin, we rescale the interval $[0,2\pi]$ to the interval $[0,1]$, and correspondingly consider $\vartheta_j' \coloneqq \tfrac{\vartheta_j}{2\pi} $ and $x_n' \coloneqq \tfrac{1}{2\pi} \sqrt{2|V|n}$.

Recall that the discrepancy $D_N$ for a sequence $(s_1,\dots,s_N)$ over an interval $[a,b]$ is defined to be
\begin{align*}
	D_N \coloneqq \sup_{a\leq c \leq d \leq b} \left\lvert \frac{\left\lvert \{ s_1,\dots,s_N \} \cap [c,d]\right\rvert}{N} - \frac{d-c}{b-a} \right\rvert,
\end{align*}
and is a quantitative measure of how far the given sequence is from equidistribution on the interval $[a,b]$.

A result of Schoi{\ss}engeier \cite{Sch}, which follows straightforwardly from the Erd\"os--Tur\'an inequality, states that for a sequence $a\sqrt{n}$ with $a \in \R^+$ one has the bound
\begin{align}\label{eqn: discrepancy bound}
	D_N \ll O\left( N^{-\frac{1}{2}} \right).
\end{align}

Now we take the sequence $(x_1', \dots x_N')$ along with $a=0,b=1$. Place an interval  $I_j$ of length $e^{-\kappa\sqrt{n} + \varepsilon\sqrt{n}}/2\pi$ centered at $\vartheta_j'$. Then using \eqref{eqn: discrepancy bound} we see that the number of points which lie in the set $[0,1) - \cup_j I_j$ is bounded below by
\begin{align}\label{eqn: lower bound on proportion}
G(n) \coloneqq 1 - O\left( n^{-\frac{1}{2}} \right).
\end{align}
 We therefore see that the proportion of values for which $\lvert V(n) \rvert \to \infty$ is at least $G(n)$. Along with the fact that $\lim_{n \to\infty} G(n) = 1$, this proves a refined version of Andrews' Conjecture 3.

Using \eqref{eqn: lower bound on proportion} along with the fact that $\mathcal{F}_\pm (\sqrt{2|V|(n+j)}) = \mathcal{F}_\pm(\sqrt{2|V|n}) + O(n^{-\frac{1}{2}})$ by \eqref{eqn: trig bound}, it is clear that almost all $4$-tuples $V(n), V(n+1), V(n+2), V(n+3)$ will all have an exponentially dominant main term $M$. In turn, this means that we automatically obtain the two plus and two minus signs almost always.

\section{Andrews' Conjectures 5 and 6}\label{Sec: explanations}

In this section we discuss Andrews' Conjectures~5 and~6 regarding the coefficients $V_1(n)$. While our methods below do not lead to complete proofs of these two conjectures, they do explain them, ultimately relating $V_1(n)$ to the arithmetic of $\mathbb Q(\sqrt{-3})$.  We restate these conjectures (as in Section~\ref{sec_intro}) here for convenience.
\begin{conjA}[Conjecture 5 \cite{Andrews86}]  For $n\geq 5$ there is an infinite sequence $N_5=293, N_6=410, N_7=545, N_8=702,\dots,N_n \geq 10n^2,\dots$ such that $V_1(N_n), V_1(N_n+1), V_1(N_n+2)$ all have the same sign.
\end{conjA}
\begin{conjA}[Conjecture 6 \cite{Andrews86}]  The numbers $|V_1(N_n)|$, $|V_1(N_n+1)|$, $|V_1(N_n+2)|$ contain a local minimum of the sequence $|V_1(j)|$.   
\end{conjA}

Conjecture 6 of Andrews is seen to be essentially explained by Conjecture 5 alongside the asymptotic of $V_1(n)$ given by our Theorem \ref{Conj3}, as it is apparent that for a sign pattern disruption, one must have a local minimum of the sequence $|V_1(j)|$. There are several possibilities for how the sign pattern may fail. Each would rely on determining more concrete information on the error term $E(n)$, which we discuss at the end of this section.

	We remark that for the sign pattern change, one needs that the main term is arbitrarily small infinitely often. A natural related question to pose is whether in fact $V_1(n) = 0$ for infinitely many $n$ in analogy to $\sigma$. Numerical computations suggest that $V_1(n)$ only vanishes for finitely many values of $n$. Checking the first five million coefficients of $v_1(q)$ it appears that $V_1(n)=0$ if and only if
	\bea
	n\in \{2,\,
	3,\,
	4,\,
	5,\,
	11,\,
	13,\,
	15,\,
	17,\,
	19,\,
	21,\,
	25,\,
	29,\,
	31,\,
	39,\,
	47,\,
	58,\,
	60,\,
	62,\,
	64,\,
	101,\,
	111,\,
	123,\,
	129,\,
	198\}.
	\eea
	Our techniques are not amenable to proving that only a finite number of coefficients of $v_1$ vanish since, as discussed below, one would require irrationality results on $\pi^2/|V|$.

\subsection{$M(n)$ and $E(n)$ have the same sign}

Assume that $M(n)$ and $E(n)$ have the same sign. Then we see that in order for each of Andrews' Conjecture 5 and 6 to hold, a necessary (but not sufficient) condition is that the main term $M(n)$ must be arbitrarily small infinitely often.
In turn, we want to find infinite families of integers that are close to the roots $\vartheta_j$. Solving directly, we want to choose infinitely many $n \in \N$ to be arbitrarily close to
\begin{align*}
	\frac{\pi^2\left(\ell \pm \frac{1}{4}\right)^2}{2 |V|}, \qquad \ell \in \Z.
\end{align*} 
To determine whether such choices exist requires more concrete knowledge regarding $\frac{\pi^2}{|V|}$. By results of Milnor \cite{Mil} we have that
\begin{align*}
	\lvert V \rvert \= \frac{9\sqrt{3} \zeta_{\Q(\sqrt{-3})} (2) }{16\pi^2},
\end{align*}
where $\zeta_K$ is the usual Dedekind zeta function associated with the field $K$.

The question of whether a given value of a Dedekind zeta value, say $\zeta_K(2)$ with $K$ a number field, is rational or irrational is a particularly deep question that has been investigated by many authors. In the case where $K$ is totally real, Klingen \cite{Kling} and Siegel \cite{Siegel2} used powerful techniques within the theory of Hilbert modular forms to provide the celebrated Siegel--Klingen theorem\footnote{The case where $K$ is real quadratic was already known to Hecke \cite{Hecke}.}, see e.g. \cite{Siegel}, which states that the values $\zeta_{K}(2n) \in | \text{disc}(K)|^{-\frac{1}{2}} \pi^{2kN} \Q$ with $N = [K:\Q]$ and $n \in \N$. 

However, we are interested in the case where $K$ is imaginary quadratic, and thus lie outside of the scope of Siegel--Klingen. In fact, current methods are unable to determine an analogue of Siegel--Klingen for imaginary quadratic fields, and we reach an impasse. Zagier \cite{Zagier} investigated the values $\zeta_K(2)$ for arbitrary number fields, and determined a representation for them as a multiple of powers of $\pi$, $\sqrt{\text{disc}(K)}$ and integrals of the shape
\begin{align*}
	A(x) \= \int_0^x \frac{1}{1+t^2} \log \frac{4}{1+t^2} dt 
\end{align*}
evaluated at certain points (see \cite[Theorem 1]{Zagier}). In fact, for imaginary quadratic fields, Zagier gave a refined sharper theorem in \cite[Theorem 3]{Zagier}. Despite these beautiful results, we are still unable to determine rationality properties of the zeta values. 

A further example of the depth of such questions is that of the algebraic dependence of $\log(2),$ $\pi,$ $\zeta(3),$ with $\zeta$ the usual Riemann zeta function. This was originally conjectured by Euler \cite{Euler} in 1785. Very recently, Eskandari and Murty \cite{Esk} determined a certain motive with periods given precisely by these three values (along with a fourth period). Conditional on the Grothendieck conjecture, this then proved that in fact this triple is algebraically independent (in opposition to Euler's conjecture). This perhaps leads to a pathway to (conditionally) prove that $\zeta_K(2),\pi,\sqrt{\text{disc}(K)}$ for $K$ imaginary quadratic are algebraically independent by constructing a motive with these periods in much the same spirit as \cite{Esk}.

Thus asking for rationality properties of 
\begin{align*}
	\frac{\pi^2}{|V|} \= \frac{16 \pi^4}{ 9\sqrt{3} \zeta_{\Q(\sqrt{-3})} (2) }
\end{align*}
seems out-of-reach of current mathematics. However, we are able to make slight progress in determining whether there are infinitely many values $n$ such that the main term $M(n)$ is arbitrarily small if we make assumptions on $\frac{\pi^2}{|V|}$.

We consider three disparate cases, depending on the nature of $\frac{\pi^2}{|V|}$. In the case of $\frac{\pi^2}{|V|}$ being irrational, we show that the main term is arbitrarily small infinitely often. In the case where $\frac{\pi^2}{|V|}$ is rational with odd denominator, we show that the main term $M(n)$ vanishes infinitely often, while in the (unlikely) case where $\frac{\pi^2}{|V|}$ is rational with even denominator we show that $M(n)$ cannot be arbitrarily small infinitely often.

\subsubsection{Case I: $\frac{\pi^2}{|V|}$ is irrational}
Assume that $\frac{\pi^2}{|V|}$ is irrational. Then we want to determine whether there are infinitely many choices of positive integers $\ell,n$ such that
\begin{align*}
	\frac{2n}{\left(\ell \pm \frac{1}{4}\right)^2} \= \frac{32n}{(4\ell \pm 1)^2} 
\end{align*}
is arbitrarily close to $\frac{\pi^2}{|V|}$. 

We use \cite[Theorem 1]{BH} with $\beta =0$, which then reads as follows, using the notation $\Vert \cdot \Vert$ to denote the distance to the nearest integer.
\begin{theorem}\label{Thm: BH}
	Let $\alpha$ be irrational and $k \geq 1$. Then there are infinitely many primes $p$ such that
	\begin{align*}
		\Vert \alpha p^k \Vert < p^{-\rho(k) + \varepsilon}
	\end{align*}
	for every $\varepsilon >0$, where $\rho(2) = \frac{3}{20}$ and $\rho(k) = (3\cdot 2^{k-1})^{-1}$ for $k\geq 3$.
\end{theorem}
Applying Theorem \ref{Thm: BH} to the irrational $\frac{\pi^2}{32|V|}$ yields infinitely primes $p$ and integers $n$ such that
\begin{align*}
	\left\lvert \frac{\pi^2}{32|V|} p^2 - n \right\rvert < p^{-\frac{3}{20}+ \varepsilon}.
\end{align*}
Dividing both sides by $p^2$ gives
\begin{align*}
	\left\lvert \frac{\pi^2}{32|V|} - \frac{n}{p^2} \right\rvert < p^{-\frac{43}{20}+ \varepsilon}.
\end{align*}
Now simply noting that all primes $\neq 2$ are of the form $4\ell \pm 1$, we see that we obtain infinitely many integer pairs $(\ell,n)$ such that we are arbitrarily close to roots $\vartheta_j$, and thus infinitely many pairs where the main term is arbitrarily small (by taking large enough $p$).

We see that in Case I, as one transitions through a zero of $\mathcal{F}_\pm$, the sign pattern of $M(n)$ is naturally disrupted. Moreover, since $E(n)$ is assumed to have the same sign as $M(n)$, these are the only places that such a transition should occur. We observe an infinite sequence of zeros of the trigonometric function $\mathcal{F}_\pm$, which we label $\vartheta_j$. Taking the sequence as 
\begin{align}\label{eqn: transition points}
	\lfloor \vartheta_j \rfloor
\end{align}
 yields an infinite sequence of integers around which one would expect to have three of the same sign, as predicted by Andrews. Testing numerically, of the first $715$ values where $V_1$ has three values with the same sign (i.e.\@ testing the first five million coefficients of $v_1(q)$), the sequence constructed in \eqref{eqn: transition points} is always within $2$ of the sequence conjectured by Andrews. This may be explained by the fact that the sequence of consecutive values with the same sign need not begin precisely at $\lfloor \vartheta_j \rfloor$, but could begin up to two terms before this value. In Table~\ref{Table1} we give the first ten values of $N_j$ for $j \geq 5$ as predicted by Andrews alongside the values of $\lfloor \vartheta_j \rfloor$.
 
\begin{table}[h]
\begin{tabular}{|c|c|c|c|c|c|c|c|c|c|c|}
	\hline
	\boldmath{$j$} & 5  & 6 & 7 & 8 & 9 & 10 & 11 & 12 & 13 & 14 \\
	\hline
	\boldmath{$N_j$} & 293 & 410 & 545 & 702 & 877 & 1072 & 1285 & 1518 & 1771 & 2044\\
	\hline
	\boldmath{$\lfloor \vartheta_j \rfloor$}& 294 & 410 & 546 & 702 & 877 & 1072 & 1286 & 1519 & 1772 & 2044 \\
	\hline
\end{tabular}
\caption{The first $10$ values of $N_j$ and $\lfloor \vartheta_j \rfloor$ for $j \geq 5$.}
\label{Table1}
\end{table}

\subsubsection{Case II: $\frac{\pi^2}{32 |V|}$ is rational with odd denominator}

Assume that $\frac{\pi^2}{32 |V|} = \frac{h}{k} \in \Q$ with $\gcd(h,k)=1$ and $k$ odd. We see that one would need to choose infinitely many positive integers $n$ that are arbitrarily close to the points
\begin{align}\label{eqn: rational}
	\frac{h}{k} \left(4\ell \pm 1\right)^2.
\end{align}
This is clearly true infinitely often, in particular when $(4\ell \pm 1)^2 = \alpha k$ with $\alpha \in \Z$. In turn, this means that the trigonometric function $\mathcal{F}_\pm$ arising in the main term would be evaluated precisely at one of its roots, meaning that the main term $M(n) =0$ at these points. 

However, this case would fail to explain Andrews' Conjecture 5, since $M(n)$ and $E(n)$ were assumed to have the same sign, and so numerically it seems implausible for Case II to hold.

\subsubsection{Case III: $\frac{\pi^2}{32 |V|}$ is rational with even denominator}

As in Case II, we assume that $\frac{\pi^2}{32 |V|} = \frac{h}{k} \in \Q$ with $\gcd(h,k)=1$. However, if $k$ is even, then it is clear that the right-hand side of \eqref{eqn: rational} has fixed denominator $k$, and thus there cannot be infinitely many integers arbitrarily close to such points.

Based on numerical evidence, the sequence $N_j$ of places where $V_1(n)$ contain three consecutive terms with the same sign appears to be infinite, beginning with the values given in Table \ref{Table1}. In turn, this provides strong evidence that one may discount Case III.

\subsection{$M(n)$ and $E(n)$ have opposite signs}
Now assume that $M(n)$ and $E(n)$ arise with different signs. Then when $\mathcal{F}_\pm$ is of the order $e^{-\kappa\sqrt{n}}$ the main term $M(n)$ and error term $E(n)$ have the same order of growth. Labelling these points $N_j$ this means that there would be a sequence of points $n<N$ where $M(n)$ determines the sign of $V_1(n)$, but as $n$ approaches the value $N_j$ the main term and error term become close to one another. As one passes the point $N$, the sign of $V_1(n)$ is then dictated by $E(n)$, which dominates $M(n)$ when $\mathcal{F}_\pm (n) < e^{-\kappa\sqrt{n}}$. Similarly, there will be such transition points where $M(n)$ again begins to dominate $E(n)$. One sees that the sign pattern of two plus signs and two minus signs would be disrupted at these transition points, providing an infinite sequence of points explaining Andrews' Conjecture 5. However, these points would not be those given in \eqref{eqn: transition points} above, and so numerically it appears that this is not the case.

Overall, we see that in order to more clearly determine information on Andrews' Conjecture 5 and 6, one needs much more precise information on the error term $E(n)$. Using Theorem \ref{Thm: radial asymp} it may be possible to keep track of all order $4m$ roots of unity in a similar fashion to how the problem is handled here. In fact, in a future project it is planned that we use the full Hardy--Ramanujan Circle Method to determine much more precise asymptotics for $V_1(n)$, which should yield an asymptotic as an infinite sum over terms of a similar shape to $M(n)$. Several obstacles would then remain. In particular, the barrier of not knowing the nature of irrationality of $\frac{\pi^2}{|V|}$ remains, and at best one would be able to prove conditional theorems.

\section*{Data availability statement} Data sharing not applicable to this article as no datasets were generated or analyzed during the current study.

\section*{Conflict of interest statement}
On behalf of all authors, the corresponding author states that there is no conflict of interest.

\end{document}

%% file: contour2.tex


\centering 
\tikzmath{
\x1 = {sqrt(2+sqrt(2))/2};
\y1 ={sqrt(2-sqrt(2))/2}; 
\d0 = .25;
\Rr = 3.5;
\RI = 1.75; \RIV = 1.25; 
}
	\begin{tikzpicture}

		\draw[gray, thick] ({\RI*\x1},{\RI*\y1}) arc [start angle=22.5, delta angle=-22.5, radius=\RI];
		\node[gray] at (2.55,0.5) {\tiny$\frac \pi 4 - \frac {\arg\varphi} 2-\varepsilon$};
		\draw[gray, thick] ({\RIV*\y1},{-\RIV*\x1}) arc [start angle=-67.5, delta angle=67.5, radius=\RIV];
		\node[gray] at (1.65,-1.25) {\tiny $\frac \pi 4 + \frac  {\arg\varphi} 2-\varepsilon$};

		\draw[black, thick,->] (-4,0) -- (4,0)	node[anchor=west]{$\R$};
		\draw[black, thick,->] (0,-4) -- (0,3)	node[anchor=south]{$i\R$};

		\filldraw [red] (0,0) circle (.08);
		\filldraw [red] (1,0) circle (.08);
		\filldraw [red] (2,0) circle (.08);
		\filldraw [red] (3,0) circle (.08);
		\filldraw [red] (-1,0) circle (.08);
		\filldraw [red] (-2,0) circle (.08);
		\node[red, anchor = south] at (-2,.25) {poles at $2\Z$};
		\filldraw [red] (-3,0) circle (.08);
		\filldraw [red] (-4,0) circle (.08);

		\draw[orange, very thick] ({\Rr*\x1},{\Rr*\y1}) arc [start angle=22.5, delta angle=-90, radius=\Rr];
		\node[orange] at (3,-1) {$C_R$};

		\draw[blue, very thick] ({(\Rr+.02)*\y1},{-(\Rr+.02)*\x1}) -- ({(\d0-0.02)*\y1},{-(\d0-0.02)*\x1});
		\draw[blue, very thick] ({(\d0-.02)*\x1},{(\d0-.02)*\y1}) -- ({(\Rr+.02)*\x1},{(\Rr+.02)*\y1});
		\node[blue] at (2,1.5) {$L_R$};

		\draw[blue, very thick] ({\d0*\x1},{\d0*\y1}) arc [start angle=22.5, delta angle=270, radius=.25];

	\end{tikzpicture}

%% file: contourihLinfty.tex
\centering
\tikzmath{
\x1 = {sqrt(2+sqrt(2))/2};
\y1 = {sqrt(2-sqrt(2))/2}; 
\d0 = .25;
\Rr = 4;
\RI = 1.75; \RIV = 1.25; 
}
\begin{tikzpicture}

		\draw[gray, thick] ({-\RI*\y1},{-\RI*\x1}) arc [start angle=-22.5-90, delta angle=22.5, radius=\RI];
		\node[gray] at (-1.75,-1) {$\frac \pi 4 - \frac {\arg\varphi} 2 - \varepsilon$};
		\draw[gray, thick] ({\RIV*\x1},{-\RIV*\y1}) arc [start angle=-22.5, delta angle=-67.5, radius=\RIV];
		\node[gray] at (1.75,-1.3) {$\frac \pi 4 + \frac {\arg\varphi} 2 - \varepsilon$};

		\draw[black, thick,->] (-4,0) -- (4,0)	node[anchor=west]{$\R$};
		\draw[black, thick,->] (0,-4) -- (0,4)	node[anchor=south]{$i\R$};

		\filldraw [red] (0,0) circle (.08);
		\filldraw [red] (-\y1,\x1) circle (.08);
		\filldraw [red] (-2*\y1,2*\x1) circle (.08);
		\filldraw [red] (-3*\y1,3*\x1) circle (.08);
		\filldraw [red] (-4*\y1,4*\x1) circle (.08);
		\filldraw [red] (1*\y1,-1*\x1) circle (.08);
		\filldraw [red] (2*\y1,-2*\x1) circle (.08);
		\filldraw [red] (3*\y1,-3*\x1) circle (.08);
		\filldraw [red] (4*\y1,-4*\x1) circle (.08);
		\node[red] at (-3,3) {poles at $2iz\Z$};

		\draw[blue, very thick] ({-\y1/\x1*4},{-4}) -- ({-(\d0-.02)*\y1},{-(\d0-.02)*\x1});
		\draw[blue, very thick,>-] ({-\y1/\x1*4/2},{-4/2}) -- ({-(\d0-.02)*\y1},{-(\d0-.02)*\x1});

		\draw[blue, very thick] ({(\d0-.02)*\x1},{-(\d0-.02)*\y1}) -- ({4},{-\y1/\x1*4});
		\draw[blue, very thick,->] ({(\d0-.02)*\x1},{-(\d0-.02)*\y1}) -- ({4/2},{-\y1/\x1*4/2});

		\node[blue] at (2.5,-.5) {$-i\varphi L_\infty$};

		\draw[blue, very thick] ({\d0*\x1},{-\d0*\y1}) arc [start angle=-22.5, delta 
angle=270, radius=.25];

\end{tikzpicture}

%% file: contourS.tex
\centering 
\tikzmath{
\x1 = {sqrt(2+sqrt(2))/2};
\y1 ={sqrt(2-sqrt(2))/2}; 
\d0 = .25;
\Rr = 3.5;
\RI = 1.75; \RIV = 1.25; 
}
\begin{tikzpicture}

		\draw[black, thick,->] (-4,0) -- (4,0)	node[anchor=west]{$\R$};
		\draw[black, thick,->] (0,-4) -- (0,4)	node[anchor=south]{$i\R$};

		\filldraw [red] (0,0) circle (.08);
		\filldraw [red] (-\y1,\x1) circle (.08);
		\filldraw [red] (-2*\y1,2*\x1) circle (.08);
		\filldraw [red] (-3*\y1,3*\x1) circle (.08);
		\filldraw [red] (-4*\y1,4*\x1) circle (.08);
		\filldraw [red] (1*\y1,-1*\x1) circle (.08);
		\filldraw [red] (2*\y1,-2*\x1) circle (.08);
		\filldraw [red] (3*\y1,-3*\x1) circle (.08);
		\filldraw [red] (4*\y1,-4*\x1) circle (.08);
		\node[red] at (-3,3) {poles at $2iz\Z$};

\draw[blue, very thick,-] 
({-\y1/\x1*(2-2/3)},4/3) -- (-\y1/\x1*2,0);
\draw[blue, very thick,<-] 
(-\y1/\x1*2,0) -- (-4*\y1/\x1,-4);

\draw[blue, very thick,->] 
({-\y1/\x1*(2-2/3)},4/3) -- (\y1/\x1*2,0);
\draw[blue, very thick,] 
(\y1/\x1*2,0) -- (4,{4/5 - 8*\x1/\y1/5)});
\node[blue] at (-1.5,-1) {$\mathcal S$};

\end{tikzpicture}

%% file: contourS1.tex
\centering 
\tikzmath{
\x1 = {sqrt(2+sqrt(2))/2};
\y1 ={sqrt(2-sqrt(2))/2}; 
\d0 = .25;
\Rr = 3.5;
\RI = 1.75; \RIV = 1.25; 
}

\begin{tikzpicture}

		\draw[black, thick,->] (-4,0) -- (4,0)	node[anchor=west]{$\R$};
		\draw[black, thick,->] (0,-4) -- (0,4)	node[anchor=south]{$i\R$};

	\draw[red, very thick] (0,0) -- (4*\y1/\x1,-4);
	\draw[red, very thick] (0+3,0) -- ({(-4*\y1+3)/\x1},4);
	\draw[red, very thick] (0-3,0) -- ({(-4*\y1-3)/\x1},4);

	\node[red, anchor=south] at ({(-4*\y1+3)/\x1},4) {\tiny{$\re((v-\frac \pi2)/\varphi) = 0$}};
	\node[red, anchor=south] at ({(-4*\y1-3)/\x1},4) {\tiny{$\re(v+\frac \pi2 )/\varphi) = 0$}};
	\node[red, anchor=south west] at (4*\y1/\x1,-4) {\tiny{$\re(v/\varphi) = 0$}};

\draw[blue, very thick,-] 
({-\y1/\x1*(2-2/3)},4/3) -- (-\y1/\x1*2,0);
\draw[blue, very thick,<-] 
(-\y1/\x1*2,0) -- (-4*\y1/\x1,-4);

\draw[blue, very thick,->] 
({-\y1/\x1*(2-2/3)},4/3) -- (\y1/\x1*2,0);
\draw[blue, very thick,] 
(\y1/\x1*2,0) -- (4,{4/5 - 8*\x1/\y1/5)});
\node[blue, anchor=south west] at (2,-1) {$\gamma'$};
\node[blue, anchor=south west] at (-1.5,-1) {$\gamma$};

\filldraw [black] (-\y1/\x1*2,0) circle (.06) node[anchor=south east] {$-\frac \pi {12}$};
\end{tikzpicture}

%% file: FMRS_arXivVersion26.bbl
\begin{thebibliography}{99}
	\bibitem{AndrewsDragonette} G. E. Andrews,.\textit{On the theorems of Watson and Dragonette for Ramanujan’s mock theta functions}, Am. J. Math. 88 (1966), 454--490.
	\bibitem{AndrewsLostIV} G.E. Andrews, \textit{Ramanujan's ``Lost" notebook IV.  Stacks and alternating parity in partitions,} Adv. in Math. 53 (1984), no. 1, 55--74.
	\bibitem{Andrews86} G.E. Andrews, \textit{Questions and conjectures in partition theory,} Amer. Math. Monthly  93(9) (1986), 708-711.
	\bibitem{AndrewsAdv} G.E. Andrews, \textit{Ramanujan's ``Lost" notebook. V. Euler's partition identity,}
Adv. in Math. 61 (1986), no. 2, 156--164.
\bibitem{AndrewsThy} G.E. Andrews, \textit{The theory of partitions,}  
Reprint of the 1976 original. Cambridge Mathematical Library. Cambridge University Press, Cambridge, 1998. xvi+255 pp.  
	\bibitem{ADH} G.E. Andrews, F.J. Dyson, and D. Hickerson, \textit{Partitions and indefinite quadratic forms,}
Invent. Math. 91 (1988), no. 3, 391--407. 
	\bibitem{BH} R.C. Baker and G. Harman, \textit{On the distribution of $\alpha p^k$ modulo one}, Mathematika 38 (1991), no. 1, 170–184.

	\bibitem{BJSM} K. Bringmann, C. Jennings-Shaffer, and K. Mahlburg, \textit{On a Tauberian theorem of Ingham and Euler–Maclaurin summation}, Ramanujan J. 61 (2023), 55-86.
	\bibitem{BringmannLovejoyRolen} K. Bringmann, J. Lovejoy, and L. Rolen, {\it On Some Special Families of $q$-hypergeometric Maass Forms}, Int. Math. Res. Not. IMRN (2018)  18, p. 5537-5561.
	\bibitem{BO} K. Bringmann and K. Ono, \textit{The $f(q)$ mock theta function conjecture and partition ranks}, Invent. Math. 165 (2006), no.2, 243--266.
	\bibitem{CGZ} F. Calegari, S. Garoufalidis, and D. Zagier, \textit{Bloch groups, algebraic $\operatorname K$-theory, units and Nahm's Conjecture,} Ann. Sci. Éc. Norm. Supér. (4) 56 (2023), no. 2, 383–426.
	\bibitem{Cohen} H. Cohen, \textit{$q$-identities for Maass waveforms,} Invent. Math. 91 (1988), no. 3, 409--422.
	\bibitem{colin} C.C. Adams, \textit{The noncompact hyperbolic 3-manifold of minimal volume,} Proc. Amer. Math. Soc. 100 (1987), no. 4, 601--606.

\bibitem{Drag} L. Dragonette, \textit{Some asymptotic formulae for the mock theta series of Ramanujan}, Trans. Am. Math. Soc. 72 (1952), 474--500.
	\bibitem{ET} P. Erd\"{o}s and P. Turán, \textit{On a problem in the theory of uniform distribution}, I. Nederl. Akad. Wetensch., Proc. 51 (1948), 1146--1154.
	\bibitem{Esk} P. Eskandari and K. Murty, \textit{On unipotent radicals of motivic Galois groups}, Algebra Number Theory 17 (2023), no. 1, 165--215.
	\bibitem{Euler} L. Euler, \textit{De relatione inter ternas pluresve quantitates instituenda} (E591), Opusc. Anal. 2 (1785), 91--101. Reprinted in Opera Omnia. Ser. 1, Vol. 4: 136--145.
	\bibitem{GZ} S. Garoufalidis and D.B. Zagier, \textit{Asymptotics of {N}ahm sums at roots of unity}, Ramanujan J. 55 : 1 (2021)
	\bibitem{GZknots} S. Garoufalidis and D.B. Zagier, \textit{Knots and their related q-series}, SIGMA Symmetry Integrability Geom. Methods Appl.\ 19 (2023), Paper No. 082, 39 pp.
	\bibitem{GR} G. Gasper, and M. Rahman. \textit{Basic hypergeometric series}. Vol. 96. Cambridge university press, Vancouver (2004).	
	\bibitem{HR} G. H. Hardy and S. Ramanujan, \textit{Asymptotic formul\ae\ in combinatory analysis}, Proc. London Math. Soc. (2) 17 (1918), 75–115.
	
	\bibitem{Hecke} E. Hecke, \textit{Eine neue Art von Zetafunktionen und ihre Beziehungen zur Verteilung der Primzahlen},
	Math. Z. 1 (1918), no. 4, 357–376.
	\bibitem{Hooley} C. Hooley, \textit{On an elementary inequality in the theory of {D}iophantine approximation}, in Analytic Number Theory, Proceedings in a conference in honour of Heini Halberstam (Birkhauser 1996), 471--486.
	\bibitem{MJ} M.-J. Jang, \textit{Asymptotic behavior of odd-even partitions}, Electron. J. Combin. 24 (2017), no. 3, Paper No. 3.62, 15 pp.
	\bibitem{Kap} V. Kaplunovsky, \textit{Saddle Point Method of Asymptotic Expansion,}  \url{https://web2.ph.utexas.edu/~vadim/Classes/2022f/saddle.pdf}, unpublished notes.
	\bibitem{Kling} H. Klingen, \textit{\"{U}ber die Werte der Dedekindschen Zetafunktion}, Math. Ann. 145 (1961/62), 265–272.
	\bibitem{KrauelRolenWoodbury} M. Krauel, L. Rolen, and M. Woodbury, {\it On a Relation Between Certain $q$-hypergeometric series and Maass Waveforms,} Proc. Amer. Math. Soc. (2017) 145,  543--557.
	\bibitem{Kup} L. Kuipers and H. Niederreiter, \textit{Uniform distribution of sequences}, Pure and Applied Mathematics. Wiley-Interscience [John Wiley \& Sons], New York-London-Sydney, (1974). xiv+390 pp.
	\bibitem{LiNgoRhoades} Y. Li, H. Ngo, and R.C. Rhoades, {\it Renormalization and quantum modular forms, part I: Maass wave forms}, preprint, arxiv: arXiv:1311.3043.
	\bibitem{LR} Y. Li and C. Roehrig, {\it Mock Maass Forms Revisited}, The Quarterly Journal of Mathematics, (2025) 76, 367--380.
	\bibitem{21LovejoyCologne} J. Lovejoy {\it Parity Bias in Partitions}, Talk at the Oberseminar Zahlentheorie in Cologne (2021), \url{http://www.mi.uni-koeln.de/algebra/seminars/nt2021slides/lovejoy.pdf}
	\bibitem{Mil} J. Milnor, \textit{Hyperbolic geometry: the first 150 years}, Bull. Amer. Math. Soc. (N.S.) 6 (1982), no. 1, 9–24.
	\bibitem{Nahm} W. Nahm, \textit{Conformal Field Theory and Torsion Elements of the Bloch Group,} Frontiers in Number Theory, Physics, and Geometry, II.  67--132.  Springer, Berlin (2007).
	\bibitem{NR} H.T. Ngo and R.C. Rhoades, \emph{Integer partitions, probabilities and quantum modular forms,} Res. Math. Sci. 4: 17 (2017), 36pp.
	\bibitem{OSullivan} C. O'Sullivan, \emph{Revisiting the saddle-point method of Perron,} Pacific J. Math 298 (1) (2019), 157--199.
	\bibitem{Olver} F. Olver, \emph{Asymptotics and special functions}, CRC Press, Vancouver (1997).
	\bibitem{Pesk} M.E. Peskin and D.V. Schroeder, \emph{An introduction to quantum field theory,} Addison-Wesley, Reading, MA (1995).  xxii+842pp.
\bibitem{Rad} H. Rademacher, \textit{Topics in analytic number theory,} Edited by E. Grosswald, J. Lehner and M. Newman. Die Grundlehren der mathematischen Wissenschaften, Band 169. Springer-Verlag, New York-Heidelberg, 1973. ix+320 pp. 
	\bibitem{Sch} J. Schoi{\ss}engeier, \textit{On the discrepancy of the sequence $\alpha n^\delta$}, Acta Math. Acad. Sci. Hung. { 38} (1981), 29--43.
	\bibitem{Siegel} C.L. Siegel, \textit{Advanced analytic number theory}, 2nd ed., Tata Institute of Fundamental Research Studies in Mathematics, vol. 9, Tata Institute of Fundamental Research, Bombay, (1980).
	\bibitem{Siegel2} C.L. Siegel, \textit{Berechnung von Zetafunktionen an ganzzahligen Stellen}, Nachr. Akad. Wiss. G\"ottingen Math.-Phys. Kl. II 1969 (1969), 87--102.
	\bibitem{Sthesis}M. Storzer, \textit{$q$-Series, their Modularity, and Nahm's Conjecture}, PhD thesis, University Bonn, (2024).
	\bibitem{VZ} M. Vlasenko and S. Zwegers, \textit{Nahm's conjecture: asymptotic computations and counterexamples,} Commun. Number Theory Phys. 5 (3) (2011), 617--642.
	\bibitem{Watson1910} G.N. Watson, \textit{The continuation of functions defined by generalized hypergeometric series} Trans. Camb. Phil. Soc 21 (1910): 281-299.
	\bibitem{CWthesis} C. Wheeler, \textit{Modular $q$-difference equations and quantum invariants of hyperbolic three-manifolds}, PhD thesis, University Bonn, (2023).
	\bibitem{Wright1} E. Wright, \textit{Stacks}, Quart. J. Math. Oxford Ser. 19 (1968), no.\@ 2, 313--320. 
	\bibitem{Wright2} E. Wright, \textit{Stacks. II}, Quart. J. Math. Oxford Ser.\ 22 (1971), no.\@ 2, 107--116.
	\bibitem{Zagier} D.B. Zagier, \textit{Hyperbolic manifolds and special values of Dedekind zeta-functions}, Invent. Math. 83 (1986), no. 2, 285–301.
	\bibitem{Zdilog} D.B. Zagier, \textit{The dilogarithm function,}  Frontiers in number theory, physics, and geometry. {II}, Springer, Berlin (2007), 3-65.
	\bibitem{ZMellin}D.B. Zagier, \textit{The Mellin transform and related analytic techniques,} Quantum field theory I: basics in mathematics and physics. A bridge between mathematicians and physicists, (2006), 305--323.
	\bibitem{ZwegersWave} S.P. Zwegers, \textit{Mock Maass theta functions,}  Q. J. Math. 63 (2012), no. 3, 753--770. 
	\newcommand{\etalchar}[1]{$^{#1}$}
	\bibitem[S{\etalchar{+}}09]{sage} \emph{{S}ageMath, the {S}age {M}athematics {S}oftware {S}ystem ({V}ersion
		9.6)}, The Sage Developers, 2024, {\tt https://www.sagemath.org}.
\end{thebibliography}
